\documentclass{amsart}

\usepackage{amssymb,mathrsfs,amscd,graphicx, array}
\allowdisplaybreaks
\usepackage[nocompress, noadjust]{cite}  

\usepackage{enumitem}

\usepackage{mathtools}

\makeatletter
\newcommand*{\Relbarfill@}{\arrowfill@\Relbar\Relbar\Relbar}
\newcommand*{\xeq}[2][]{\ext@arrow 0055\Relbarfill@{#1}{#2}}
\makeatother

\usepackage{tikz}
\usetikzlibrary{matrix,arrows,cd}

\usepackage[
  pdfencoding=auto, 
  psdextra,
]{hyperref}
\usepackage{letltxmacro}
\LetLtxMacro{\oldsqrt}{\sqrt}
\renewcommand{\sqrt}[2][]{\,\oldsqrt[#1]{#2}\,}

\makeatletter
\def\@tocline#1#2#3#4#5#6#7{\relax
  \ifnum #1>\c@tocdepth 
  \else
    \par \addpenalty\@secpenalty\addvspace{#2}%
    \begingroup \hyphenpenalty\@M
    \@ifempty{#4}{%
      \@tempdima\csname r@tocindent\number#1\endcsname\relax
    }{%
      \@tempdima#4\relax
    }%
    \parindent\z@ \leftskip#3\relax \advance\leftskip\@tempdima\relax
    \rightskip\@pnumwidth plus4em \parfillskip-\@pnumwidth
    #5\leavevmode\hskip-\@tempdima
      \ifcase #1
       \or\or \hskip 1em \or \hskip 2em \else \hskip 3em \fi%
      #6\nobreak\relax
    \dotfill\hbox to\@pnumwidth{\@tocpagenum{#7}}\par
    \nobreak
    \endgroup
  \fi}
\makeatother

\usepackage{bbm}
\usepackage{stmaryrd}

\makeatletter

\def\greekbolds#1{%
 \@for\next:=#1\do{%
    \def\X##1;{%
     \expandafter\def\csname V##1\endcsname{\boldsymbol{\csname##1\endcsname}}
     }
   \expandafter\X\next;
  }
}

\greekbolds{alpha,beta,iota,gamma,lambda,nu,eta,Gamma,varsigma}

\def\make@bb#1{\expandafter\def
  \csname bb#1\endcsname{{\mathbb{#1}}}\ignorespaces}

\def\make@bbm#1{\expandafter\def
  \csname bb#1\endcsname{{\mathbbm{#1}}}\ignorespaces}

\def\make@bf#1{\expandafter\def\csname bf#1\endcsname{{\bf
      #1}}\ignorespaces} 

\def\make@gr#1{\expandafter\def
  \csname gr#1\endcsname{{\mathfrak{#1}}}\ignorespaces}

\def\make@scr#1{\expandafter\def
  \csname scr#1\endcsname{{\mathscr{#1}}}\ignorespaces}

\def\make@cal#1{\expandafter\def\csname cal#1\endcsname{{\mathcal
      #1}}\ignorespaces} 

\def\do@Letters#1{#1A #1B #1C #1D #1E #1F #1G #1H #1I #1J #1K #1L #1M
                 #1N #1O #1P #1Q #1R #1S #1T #1U #1V #1W #1X #1Y #1Z}
\def\do@letters#1{#1a #1b #1c #1d #1e #1f #1g #1h #1i #1j #1k #1l #1m
                 #1n #1o #1p #1q #1r #1s #1t #1u #1v #1w #1x #1y #1z}
\do@Letters\make@bb   \do@letters\make@bbm
\do@Letters\make@cal  
\do@Letters\make@scr 
\do@Letters\make@bf \do@letters\make@bf   
\do@Letters\make@gr   \do@letters\make@gr
\makeatother

\newcommand{\abs}[1]{\lvert #1 \rvert}
\newcommand{\zmod}[1]{\mathbb{Z}/ #1 \mathbb{Z}}

\newcommand{\fl}[1]{\left\lfloor #1 \right\rfloor}
\newcommand{\dangle}[1]{\left\langle #1 \right\rangle}
\newcommand{\ur}{{\mathrm{ur}}}
\newcommand{\wh}{\widehat}
\newcommand{\wt}{\widetilde}

\DeclareMathSymbol{\twoheadrightarrow} {\mathrel}{AMSa}{"10}

\DeclareMathOperator{\fchar}{char}

\DeclareMathOperator{\Ram}{Ram}
\DeclareMathOperator{\Nr}{Nr}

\DeclareMathOperator{\Tp}{Tp}
\DeclareMathOperator{\SG}{SG}

\DeclareMathOperator{\Emb}{Emb}

\newcommand{\sg}{\mathrm{sg}}

\newcommand{\whA}{\widehat{A}}

\newcommand{\whF}{\widehat{F}}

\newcommand{\wcO}{\widehat{\mathcal{O}}}
\newcommand{\wbZ}{\widehat{\mathbb{Z}}}
\newcommand{\bsh}{\backslash}

\DeclareMathOperator{\Hom}{Hom}

\DeclareMathOperator{\Gal}{Gal}
\DeclareMathOperator{\Mat}{Mat}

\DeclareMathOperator{\Tr}{Tr}
\DeclareMathOperator{\Nm}{N}  

\DeclareMathOperator{\GL}{GL}
\DeclareMathOperator{\SL}{SL}

\newcommand{\Z}{\mathbb Z}
\newcommand{\Q}{\mathbb Q}

\newcommand{\F}{\mathbb F}

\newcounter{thmcounter} 
\numberwithin{thmcounter}{section}  
\newtheorem{thm}[thmcounter]{Theorem}
\newtheorem{lem}[thmcounter]{Lemma}
\newtheorem{cor}[thmcounter]{Corollary}
\newtheorem{prop}[thmcounter]{Proposition}
\theoremstyle{definition}
\newtheorem{defn}[thmcounter]{Definition}
\newtheorem{ex}[thmcounter]{Example}

\newtheorem{rem}[thmcounter]{Remark}

\newtheorem{que}[thmcounter]{Question}

\numberwithin{equation}{section}
\numberwithin{figure}{section}
\numberwithin{table}{section}

\newtheoremstyle{notitle}  
  {}
  {}
  {\itshape}
  {}
  {}
  {\ }
  {.5em}
  {}
\theoremstyle{notitle}

 \title[Optimal Spinor Selectivity]{Optimal spinor selectivity for quaternion
   Bass orders}
 \author{Deke Peng}

\address{(Peng) School of
  Mathematics and Statistics, Wuhan University, Luojiashan, 430072,
  Wuhan, Hubei, P.R. China}   

\email{dkpeng@whu.edu.cn}

 \author{Jiangwei Xue}

\address{(Xue) Collaborative Innovation Center of Mathematics, School of
  Mathematics and Statistics, Wuhan University, Luojiashan, 430072,
  Wuhan, Hubei, P.R. China}   

\address{(Xue) Hubei Key Laboratory of Computational Science (Wuhan
  University), Wuhan, Hubei,  430072, P.R. China.} 

\thanks{J.~Xue is partially supported
by the National Natural Science Foundation of China grant \#11601395. 
}
\email{xue\_j@whu.edu.cn}

\begin{document}
\date{\today} 

\begin{abstract}
  Let $A$ be a quaternion algebra over a number field $F$, and
  $\mathcal{O}$ be an $O_F$-order of full rank in $A$.  Let $K$ be a
  quadratic field extension of $F$ that embeds into $A$, and $B$ be an
  $O_F$-order in $K$. Suppose that $\mathcal{O}$ is a Bass order that
  is well-behaved at all the dyadic primes of $F$.  We provide a
  necessary and sufficient condition for $B$ to be optimally spinor
  selective for the genus of $\mathcal{O}$. This partially generalizes previous results
  on optimal (spinor) selectivity by C.~Maclachlan [Optimal embeddings
  in quaternion algebras. J. Number Theory, 128(10):2852–-2860, 2008]
  for Eichler orders of square-free levels,  and independently by
  M.~Arenas et~al.~[On optimal embeddings and trees. J. Number Theory,
  193:91--117, 2018] and by J.~Voight [Chapter 31, Quaternion
  algebras, volume 288 of Graduate Texts in
  Mathematics. Springer-Verlag, 2021] for Eichler orders of arbitrary
  levels. 
\end{abstract}

\maketitle



\section{Introduction}


Let $F$ be a number field, and $O_F$ be its ring of integers. Let  $A$
be a quaternion $F$-algebra. Two orders\footnote{All orders
  considered in this paper are $O_F$-orders of \emph{full rank} in their ambient
  $F$-algebra upon introduction. For example, both $\calO$ and
  $\calO'$ are of full rank in $A$.} $\calO$ and
  $\calO'$ in $A$ are said to be in the same \emph{genus} if their
  $\grp$-adic completions $\calO_\grp$ and $\calO'_\grp$ are isomorphic at
  every finite prime $\grp$ of $F$.  For example, given a fixed nonzero integral
ideal $\grn\subseteq O_F$ coprime to the reduced discriminant
$\grd(A)$ of $A$, all Eichler orders (i.e.~intersections of two maximal orders)
of level $\grn$ in $A$ form a single genus. In particular, all maximal
orders (which are simply Eichler orders of level $\grn=O_F$)  form a single
genus. Fix a genus $\scrG$ of orders in $A$.
Let $K/F$ be a quadratic field extension that embeds into $A$, and $B$ be an order  in
$K$.  The \emph{selectivity question} for $\scrG$ and $B$ can be stated as
follows:
\begin{que}\label{que:selectivity}
Whether and when can $B$ embed into every order in $\scrG$?
\end{que}
If $B$  embeds into some but not all members of $\scrG$, then
we say that $B$ is \emph{selective} for $\scrG$. In such a case, one
further asks how to determine the members of $\scrG$ that do admit an
embedding of $B$.

The selectivity question has several variants. Recall that an
\emph{optimal embedding} of $B$ into  $\calO$  is an
embedding $\varphi: K\to A$ such that
$\varphi(K)\cap \calO=\varphi(B)$.  If we substitute the word
``embed'' by ``optimally embed'' in
Question~\ref{que:selectivity}, then we get the \emph{optimal
  selectivity question}.  Generally, the (optimal) selectivity
question admits a satisfactory answer only if $A$ satisfies the
Eichler condition, that is, $A$ is split at an infinite place of $F$. Indeed, almost all literature
\cite{Chinburg-Friedman-1999,Guo-Qin-embedding-Eichler-JNT2004,Chan-Xu-Rep-Spinor-genera-2004,M.Arenas-et.al-opt-embed-trees-JNT2018,Arenas-Carmona-2013,Arenas-Carmona-Max-Sel-JNT-2012,Linowitz-Selectivity-JNT2012,voight-quat-book,Maclachlan-selectivity-JNT2008}
on (optimal) selectivity
assumes the Eichler condition.

If the Eichler condition fails, then $F$ is necessarily a totally real
field, and $A$ is ramified at all the infinite places of $F$. Such
quaternion algebras are called \emph{totally definite}.  For arbitrary
quaternion algebras including the totally definite ones, we can formulate a general notion called
\emph{optimal spinor selectivity} as in Definition~\ref{defn:OSS}. If
$A$ satisfies the Eichler condition, then ``optimal spinor
selectivity'' reduces to ``optimal selectivity'' by
Remark~\ref{rem:spinor-genus=type}. It was observed in
\cite{xue-yu:spinor-class-no} that understanding optimal spinor
selectivity plays a crucial role in computing certain class numbers
attached to orders in totally definite quaternion algebras. For this
reason, we focus on optimal spinor selectivity in this paper. 


The study on selectivity questions was initiated by Chevalley
\cite{Chevalley-matrices-1936}. Modern research on this topic
is heavily influenced by the work of Chinburg and Friedman
\cite{Chinburg-Friedman-1999}, which gives a complete answer to
Question~\ref{que:selectivity} for the genus of maximal orders in $A$.
Indeed, it was them who coined the term ``selectivity''.
Independently, Guo and Qin \cite{Guo-Qin-embedding-Eichler-JNT2004}
and Chan and Xu \cite{Chan-Xu-Rep-Spinor-genera-2004} generalized the
result to Eichler orders.  Arenas-Carmona \cite{Arenas-Carmona-2013}
and Linowitz \cite{Linowitz-Selectivity-JNT2012} obtained selectivity
theorems for more general classes of orders.  The solution to the selectivity question has important 
applications in the construction of isospectral non-isometric
hyperbolic manifolds; see \cite[\S12.4--5]{hyperbolic-3-mfld} and
\cite{Linowitz-Voight-isospectral}. 
More broadly, selectivity
results in the context of $A$ being a  central simple $F$-algebra have been
obtained by Linowitz and Shemanske \cite{Linowitz-Shemanske-2012} and
Arenas-Carmona \cite{Arenas-Carmona-Spinor-CField-2003,
  Arenas-Carmona-cyclic-orders-2012,
  Arenas-Carmona-Max-Sel-JNT-2012}.



As for optimal selectivity, Maclachlan 
\cite{Maclachlan-selectivity-JNT2008} first obtained a theorem for Eichler orders of
square-free levels. Independently, Arenas et~al.\
\cite{M.Arenas-et.al-opt-embed-trees-JNT2018} and Voight
\cite[Chapter~31]{voight-quat-book} removed the square-free condition
and obtained theorems for Eichler orders of
arbitrary levels. Their results have been generalized by Chia-Fu Yu and the
second named author to quaternion
orders with nonzero Eichler invariants at all finite primes of $F$ in
\cite[\S2.2]{xue-yu:spinor-class-no}. See
Definition~\ref{defn:eichler-invariant} for the notion of the Eichler
invariant of a quaternion order at a prime $\grp$ of $F$. Since quaternion orders with nonzero Eichler invariants
everywhere are Bass by \cite[Corollary~2.4 and
Proposition~3.1]{Brzezinski-1983}, the next natural step is  to
consider arbitrary quaternion Bass  orders. In this
paper, we study optimal spinor selectivity under the assumption that
the genus 
$\scrG$ consists of Bass orders  well-behaved (to be made
precise in (\ref{eq:4})) at the dyadic primes.  The main result will be stated in
Theorem~\ref{thm:main}. 
  See \cite[\S37]{curtis-reiner:1} for the general
theory of Bass orders in finite dimensional $F$-algebras.

This paper is organized as follows. We introduce the preliminary
notions and state our main theorem in \S\ref{sec:main-theorem}.  The
definition and basic properties of quaternion Bass orders will be
recalled in \S\ref{sec:bass-orders}. The proof of the main theorem
will be carried out in three steps in \S\ref{sec:proof-main}.  We
construct an interesting family of examples in \S\ref{sec:an-exam}.


\textbf{Notation.} Throughout this paper, $F$ is either a number
field or a nonarchimedean local field of characteristic not equal to
$2$.  We will always fix  a
quaternion $F$-algebra $A$, and write $\calO$ for an $O_F$-order in $A$. When $F$ is a number field, we write $\Ram(A)$ for the
finite set of places of $F$ that are ramified  in $A$, and
$\Ram_\infty(A)$ (resp.~$\Ram_f(A)$) for the set of the infinite (resp.~finite) ramified
places. If $\grp$ is a finite prime of $F$ and $M$ is a finite
dimensional $F$-vector space or a finite
$O_F$-module, we write $M_\grp$ for the $\grp$-adic completion of
$M$. In particular, $F_\grp$ is the $\grp$-adic completion of $F$,
whose $\grp$-adic discrete valuation is denoted by
$\nu_\grp: F_\grp^\times\twoheadrightarrow \Z$.  Let
$\wbZ=\varprojlim \zmod{n}=\prod_p \Z_p$ be the profinite completion
of $\Z$. If $X$ is a finitely generated $\Z$-module or a finite
dimensional $\Q$-vector space, we set $\wh X=X\otimes_\Z\wbZ$. For
example, $\whA$ is the ring of finite adeles of $A$, and
$\wcO=\prod_{\grp}\calO_\grp$.

\section{Basic notions and the main theorem}
\label{sec:main-theorem}

In this section, we introduce some preliminary notions and state our main
result. 

By definition, two orders $\calO$ and $\calO'$ in $A$ are said to be 
in the same \emph{genus} if they are locally isomorphic
everywhere, or equivalently, if there exists $x\in \whA^\times$ such
that   $\wcO'= x \wcO x^{-1}$.  The orders $\calO$ and $\calO'$  are
said to be of the same \emph{type} if they are isomorphic, or
equivalently, if there exists $\alpha\in A^\times$ such that $\calO'= \alpha\calO
\alpha^{-1}$. Let
$[\calO]:=\{\alpha \calO \alpha^{-1}\mid \alpha\in A^\times\}$
be the type of $\calO$, and $\Tp(\calO)$ be the finite set of types of
orders in the genus of $\calO$. We regard $\Tp(\calO)$ as a
pointed set with
the base point $[\calO]$.  If $\scrG:=\scrG(\calO)$ denotes the genus of
$\calO$, then we  write $\Tp(\scrG)$ for the type set  $\Tp(\calO)$
with the base point omitted. 
The quaternion algebra
$A$ admits a canonical involution $\alpha\mapsto \bar \alpha$ such
that $\Tr(\alpha)=\alpha+\bar \alpha$ and
$\Nr(\alpha)=\alpha \bar \alpha$ are respectively the \emph{reduced trace}
and \emph{reduced norm} of $\alpha\in A$. Given a set $X\subseteq
\whA$, we write $X^1$ for the subset of elements with reduced norm
$1$, that is,  
\begin{equation}
  \label{eq:116}
X^1:=\{x\in X\mid \Nr(x)=1\}. 
\end{equation}

\begin{defn}[{\cite[\S1]{Brzezinski-Spinor-Class-gp-1983}}]\label{defn:spinor-genus-class}
Two orders $\calO$ and $\calO'$ in $A$ are said to be in the same
  \emph{spinor genus} if there
  exists  $x\in A^\times\whA^1$ such that
  $\wcO'= x \wcO x^{-1}$.  
\end{defn}
We write $\calO\sim \calO'$ if  $\calO$ and $\calO'$ are in the same
  spinor genus.   The spinor genus of
  $\calO$ is denoted by $[\calO]_\sg$. For the  genus $\scrG=\scrG(\calO)$, the set of spinor genera within $\scrG$ is denoted by
  $\SG(\scrG)$. We
  often  write $\SG(\calO)$ for 
  the pointed set $\SG(\scrG)$ with the base point
  $[\calO]_\sg$.  By definition, there is a canonical projection of pointed
  sets
  \begin{equation}\label{eq:1}
    \Tp(\calO)\twoheadrightarrow \SG(\calO), \qquad [\calO']\mapsto [\calO']_\sg.
  \end{equation}
  \begin{rem}\label{rem:spinor-genus=type}
When $A$ satisfies the
Eichler condition, Brzezinski
\cite[Proposition~1.1]{Brzezinski-Spinor-Class-gp-1983} shows that the
above map is a bijection, that is, 
each spinor genus of orders consists of exactly one type. 
  \end{rem}

Let $K/F$ be a quadratic field extension.  We assume that $K$ is
$F$-embeddable into $A$ throughout this section. In light of the Hasse-Brauer-Noether-Albert Theorem
\cite[Theorem~32.11]{reiner:mo} \cite[Theorem~III.3.8]{vigneras}, this
assumption says 
 that no place of $F$ which is 
ramified in $A$ splits in $K$.   
Given orders $B\subset K$ and $\calO\subset A$, we write $\Emb(B, \calO)$ for the set of
optimal embeddings of $B$ into $\calO$, that is
\begin{equation}
  \label{eq:22}
  \Emb(B, \calO):=\{\varphi\in \Hom_F(K, A)\mid \varphi(K)\cap \calO=\varphi(B)\}.
\end{equation}
The unit group $\calO^\times$ acts from the right on $\Emb(B, \calO)$
by  conjugation: $\varphi\mapsto u^{-1}\varphi u$ for any $u\in
\calO^\times$.  Thanks to the Jordan-Zassenhaus Theorem
\cite[Theorem~24.1, p.~534]{curtis-reiner:1}, the number of orbits
\begin{equation}\label{eq:14}
m(B, \calO, \calO^\times):=\abs{\Emb(B, \calO)/\calO^\times},   
\end{equation}
is always finite (which holds true in the local case as well). 
According to
\cite[Corollary~30.4.8]{voight-quat-book}, there exists 
$\calO'\in \scrG$ such that $\Emb(B, \calO')\neq \emptyset$ if and
only if $\Emb(B_\grp, \calO_\grp)\neq\emptyset$ for every finite prime
$\grp$ of $F$. The latter condition depends only on the
genus $\scrG$ and not on the choice of $\calO$. We define 
\begin{equation}
  \label{eq:79}
  \Delta(B, \calO)=
  \begin{cases}
    1 \qquad &\text{if } \exists\, \calO'\text{ such that $\calO'\sim \calO$
      and }\Emb(B, \calO')\neq \emptyset, \\
    0 \qquad &\text{otherwise}.
  \end{cases}
\end{equation}
Clearly, $\Delta(B,\calO)=0$ if there exists a finite prime $\grp$ of
$F$  such that $\Emb(B_\grp, \calO_\grp)=\emptyset$.  The symbol $\Delta(B, \calO)$ is featured prominently in class number formulas studied in \cite[Corollary~3.4 and
Theorem~3.8]{xue-yu:spinor-class-no}.

\begin{defn}\label{defn:OSS}
  We say $B$ is \emph{optimally spinor selective}  for  $\scrG$  if
  $\{\calO\in \scrG\mid \Delta(B, \calO)=1\}$ is a nonempty proper
  subset of $\scrG$, in which case a 
  spinor genus $[\calO]_\sg\subseteq \scrG$  with $\Delta(B, \calO)=1$
  is said to be \emph{optimally selected}
  by $B$. 
\end{defn}

From Remark~\ref{rem:spinor-genus=type}, if $A$ satisfies the Eichler
condition, then each spinor genus consists of exactly one type of orders,
and hence in this case there is no difference between ``optimal spinor
selectivity''  here and ``optimal
selectivity''  in
\cite{Maclachlan-selectivity-JNT2008,M.Arenas-et.al-opt-embed-trees-JNT2018,voight-quat-book}. 

To state our main theorem, we introduce some invariants of orders.  Given 
a finite prime $\grp$ of $F$, we write 
$\nu_\grp: F^\times \twoheadrightarrow \Z$ for the associated
normalized $\grp$-adic discrete valuation. Let $\grd(\calO)$ be the
reduced discriminant of $\calO$, and $\grf(B)$ be the conductor of
$B$, i.e.~the unique nonzero integral ideal of $F$ such that
$B=O_F+\grf(B)O_K$. Put
\begin{equation}
  \label{eq:3}
 n_\grp(\calO):=\nu_\grp(\grd(\calO)), \quad\text{and}\quad i_\grp(B):=
 \nu_\grp(\grf(B)). 
\end{equation}
  Note that  $n_\grp(\calO)=0$ if and only if
$\calO_\grp\simeq \Mat_2(O_{F_\grp})$. Similarly, $n_\grp(\calO)=1$ if
and only if one of the following is true:
\begin{itemize}
\item $A$ is split at $\grp$,  and $\calO_\grp$ is an Eichler order
  of level $\grp O_{F_\grp}$;
  \item $A$ is ramified at $\grp$, and $\calO_\grp$ is the
unique maximal order of $A_\grp$. 
\end{itemize}


\begin{defn}[{\cite[Definition~1.8]{Brzezinski-1983}}]\label{defn:eichler-invariant}
  Let   $\grk_\grp:= O_F/\grp$
be the finite residue field of $\grp$,  and $\grk_\grp'/\grk_\grp$ be the unique
quadratic field extension.  When $\calO_\grp\not\simeq
\Mat_2(O_{F_\grp})$, the quotient of $\calO_\grp$ by its Jacobson radical
$\grJ(\calO_\grp)$ falls into the following three cases: 
\[\calO_\grp/\grJ(\calO_\grp)\simeq \grk_\grp\times \grk_\grp, \qquad \grk_\grp,
\quad\text{or}\quad \grk_\grp', \]
and the \emph{Eichler invariant} $e_\grp(\calO)$ of $\calO$ at $\grp$ is defined to be
$1, 0, -1$ accordingly.  As a convention, if $\calO_\grp\simeq
\Mat_2(O_{F_\grp})$, then its Eichler invariant is defined to be
$2$.

Similarly, let $(K/\grp)$ be the
Artin symbol of $K$ at $\grp$, which takes value $1, 0, -1$ according to whether $\grp$
is split, ramified or inert in the extension $K/F$.
\end{defn}
 For
example, if $A$  is ramified at $\grp$ and $\calO_\grp$ is maximal, then
$e_\grp(\calO)=-1$. It is shown in
\cite[Proposition~2.1]{Brzezinski-1983} that 
$e_\grp(\calO)=1$ if and only if
$\calO_\grp$ is a non-maximal Eichler order (particularly,  $A$ is split
at $\grp$).  From \cite[Corollary~4.3]{Brzezinski-1983}, if $e_\grp(\calO)=0$, then
$n_\grp(\calO)\geq 2$ (see also the discussion above Definition~\ref{defn:eichler-invariant}). 


 Central to the theory of spinor
optimal selectivity is the
 class field $\Sigma_\scrG/F$ associated to the genus $\scrG$ and the map $(\calO, \calO')\mapsto
 \rho(\calO, \calO')\in \Gal(\Sigma_\scrG/F)$ on pair of orders
 $\calO, \calO'\in \scrG$. These two notions have been stable for
 almost all  variants of selectivity
 theory, cf.~\cite[\S3]{Linowitz-Selectivity-JNT2012} and
 \cite[\S31.1]{voight-quat-book}.  Following \cite[\S
 III.4]{vigneras}, we write $F_A^\times$ for the subgroup of
 $F^\times$ consisting of all elements 
that are positive at each place in  $\Ram_\infty(A)$. The Hasse-Schilling-Maass theorem \cite[Theorem~33.15]{reiner:mo}
\cite[Theorem~III.4.1]{vigneras} implies that 
\begin{equation}
  \label{eq:15}
  \Nr(A^\times)=F_A^\times.  
\end{equation}
Let $\calN(\wcO)$ be the normalizer of
 $\wcO$ in $\whA^\times$. The pointed 
 set $\SG(\calO)$ of spinor genera in $\scrG$ admits the following adelic description
(cf.~\cite[Propositions~1.2 and 1.8]{Brzezinski-Spinor-Class-gp-1983})
\begin{equation}
  \label{eq:118}
\SG(\calO)\simeq  (A^\times\whA^1)\bsh \whA^\times/\calN(\wcO)\xrightarrow[\simeq]{\Nr}
  F_A^\times\bsh \whF^\times/\Nr(\calN(\wcO)),  
\end{equation}
where the  two double coset spaces are canonically bijective via
the reduced norm map. It follows that $\SG(\calO)$ is naturally
equipped with an abelian group structure, with its distinguished point
$[\calO]_\sg$ as the
identity element.  Since $\Nr(\calN(\wcO))$ is an open subgroup of
$\whF^\times$ containing $(\whF^\times)^2$, the group $\SG(\calO)$ is a
finite elementary $2$-group
\cite[Proposition~3.5]{Linowitz-Selectivity-JNT2012}. Clearly, the
group $\Nr(\calN(\wcO))$ depends only on the genus $\scrG$ and not on
the choice of $\calO$. 

\begin{defn}[{\cite[\S2]{Arenas-Carmona-Spinor-CField-2003},
    \cite[\S3]{Linowitz-Selectivity-JNT2012}}] \label{defn:spinor-field}
  The \emph{spinor genus field} of $\scrG$ is the abelian field extension
  $\Sigma_\scrG/F$ corresponding to the open subgroup
  $F_A^\times\Nr(\calN(\wcO))\subseteq \whF^\times$ via the class
  field theory \cite[Theorem~X.5]{Lang-ANT}. 
\end{defn}
By the definition of $\Sigma_\scrG$, there are isomorphisms:
\begin{equation}
  \label{eq:8}
\SG(\calO)\simeq  F_A^\times\bsh \whF^\times/
  \Nr(\calN(\wcO))\simeq   \Gal(\Sigma_\scrG/F). 
\end{equation}
Given another order $\calO'\in \scrG$, we define $\rho(\calO, \calO')$
to be the image of $[\calO']_\sg\in \SG(\calO)$ in
$\Gal(\Sigma_\scrG/F)$ under the above isomorphism.  More canonically, we regard the base
point free set $\SG(\scrG)$ as a principal homogeneous space over
$\Gal(\Sigma_\scrG/F)$ via (\ref{eq:8}). Then $\rho(\calO, \calO')$ is
the unique element of $\Gal(\Sigma_\scrG/F)$ that sends $[\calO]_\sg$
to $[\calO']_\sg$. Clearly, $\rho(\calO, \calO')$ enjoys the following
properties:
\begin{enumerate}[label=(\alph*)]
\item $\rho(\calO, \calO')=1$ if and only if $\calO\sim \calO'$;
\item $\rho(\calO, \calO')=\rho(\calO', \calO)$;
\item   $\rho(\calO, \calO'')=\rho(\calO, \calO')\rho(\calO', \calO'')$.
\end{enumerate}

We will postpone the definition and basic properties of Bass orders to the
next section. Taking that for granted,  we are now ready to state the main theorem.
\begin{thm}\label{thm:main}
   \begin{enumerate}[label=(\Roman*)]
   \item   Let $\scrG$ be a genus of Bass orders in $A$, and $\calO$ be a
  member of $\scrG$.  Assume that $\calO$
  is well-behaved at every dyadic prime $\grq$ of $F$ in the following sense:
  \begin{equation}
    \label{eq:4}
   n_\grq(\calO)=2 \text{ if } e_\grq(\calO)=0, \quad \forall \grq|(2O_F).    
  \end{equation}
  Let  $K/F$ be a
  quadratic field extension that embeds into $A$, and
  $B$ be an order in $K$.
 Suppose that $\Emb(B_\grp, \calO_\grp)\neq \emptyset$ for every
 finite prime $\grp$
   so that $B$ is optimally embeddable into some member of
  $\scrG$. 
Then $B$ is optimally spinor selective for $\scrG$ if and only if
$K\subseteq \Sigma_\scrG$ and for every nondyadic prime $\grp$ with
$e_\grp(\calO)=0$ and $(K/\grp)=0$, one of the following 
conditions holds


  \begin{enumerate}[label=(\roman*)]
   \item  $n_\grp(\calO)\geq 2i_\grp(B)+3$; 
   \item $n_\grp(\calO)=2i_\grp(B)+1$, $A$ is split at $\grp$, and $\abs{\grk_\grp}=5$;
   \item   $n_\grp(\calO)=2i_\grp(B)+1$, $A$ is ramified at $\grp$, and $\abs{\grk_\grp}=3$. 
     \end{enumerate}
\item If $B$ is optimally spinor selective for $\scrG$, then both of
  the following hold true: 
  \begin{enumerate}
\item for any two $O_F$-orders
  $\calO, \calO'\in \scrG$,  
   \begin{equation}
\label{eq:164}
    \Delta(B, \calO')=\rho(\calO', \calO)\vert_K+\Delta(B, \calO), 
  \end{equation}
  where $\rho(\calO',
  \calO)\vert_K\in \Gal(K/F)$ denotes the restriction of $\rho(\calO',
  \calO)\in \Gal(\Sigma_\scrG/F)$ to $K$, and the
  summation is taken inside $\zmod{2}$ with the canonical
  identification $\Gal(K/F)\simeq \zmod{2}$;

  \item exactly half of the spinor genera in  $\scrG$ are
    optimally selected
    by $B$. 
  \end{enumerate}
   \end{enumerate}
\end{thm}

\begin{rem}\label{rem:special-cases}
  Suppose that $\calO$ is an order in $A$ satisfying 
  \begin{equation}\label{eq:5}
 e_\grp(\calO)\neq 0 \text{ for every finite prime } \grp \text{ of } F.     
  \end{equation}
Then $\calO$ is automatically
Bass by \cite[Corollary~2.4 and
Proposition~3.1]{Brzezinski-1983}. Moreover, the conditions 
  (\ref{eq:4}) and 
(i--iii) above are all vacuous in this case.  Therefore, if $\calO$ satisfies
condition (\ref{eq:5}), then
$B$ is optimally spinor selective for the
genus $\scrG$ if and only if $K\subseteq \Sigma_\scrG$. Since 
all Eichler orders satisfy  (\ref{eq:5}), 
we recover partial cases of
\cite[Theorem~1.1]{M.Arenas-et.al-opt-embed-trees-JNT2018} and
\cite[Theorem~31.1.7]{voight-quat-book}. On the other hand, if
condition (\ref{eq:5}) is dropped, we can easily construct examples
where $K\subseteq \Sigma_\scrG$, but $\Delta(B,\calO)=1$ for every
$\calO\in \scrG$. See \S\ref{sec:an-exam} for a
family of 
examples. 
\end{rem}


Given a Bass order  $\calO\subset A$ and an order
$B\subset K$,  in order to apply Theorem~\ref{thm:main}, a priori, we need to
know whether $\Emb(B_\grp, \calO_\grp)=\emptyset$ (equivalently, $m(B_\grp, \calO_\grp, \calO_\grp^\times)=0$)  or not for every
$\grp$.    If $e_\grp(\calO)\in \{1, 2\}$, then $A$ is split at $\grp$
 and $\calO_\grp$ is an Eichler order. In this case, the method for computing $m(B_\grp,
 \calO_\grp, \calO_\grp^\times)$ is well known and has historically  been studied
  by Eichler,
 Hijikata and many others.  See \cite[\S II.3]{vigneras} and
 \cite[\S30.6]{voight-quat-book} for some expositions.  If
 $e_\grp(\calO)\in \{-1, 0\}$, Brzezinski
 \cite{Brzezinski-crelle-1990} produced recursive formulas for
$m(B_\grp, \calO_\grp, \calO_\grp^\times)$. For example, if
$e_\grp(\calO)=0$ and $n_\grp(\calO)=2$, then $m(B_\grp, \calO_\grp,
\calO_\grp^\times)$ can be read off directly from  \cite[(3.14) and
(3.17)]{Brzezinski-crelle-1990}. See also
Corollary~\ref{cor:opt-embed-nonempty} for an application of
Brzezinski's result in the case that $\grp$ is nondyadic,
$e_\grp(\calO)=0$, and $n_\grp(\calO)\geq 3$.

From \cite[Lemma~2.8]{xue-yu:spinor-class-no}, the condition $K\subseteq \Sigma_\scrG$ can be
characterized purely in terms of local conditions. For the reader's
convenience, we recall this lemma below.  Keep in mind that $K$ is
assumed to be $F$-embeddable into $A$. 
\begin{lem}\label{lem:K-in-Sigma}
We have  $K\subseteq \Sigma_\scrG$  if and only if both of the
following conditions hold: 
\begin{enumerate}
\item[(i)] $K$ and $A$ are ramified at
exactly the same (possibly empty) set of real places of $F$;
\item[(ii)]  $\Nr(\calN(\calO_\grp))\subseteq \Nr(K_\grp^\times)$ 
 for every finite prime $\grp$ of $F$.
\end{enumerate}
\end{lem}

Thanks to  the explicit description of
normalizers of local Bass orders \cite[Theorems~2.2 and
2.5]{Brzezinski-crelle-1990}, we have the following characterization
of $\Sigma_\scrG$. 
\begin{prop}\label{prop:description-Sigma}
Let $\scrG$ and $\calO$ be as in Theorem~\ref{thm:main}.    Then $\Sigma_\scrG/F$ is
the maximal abelian extension of exponent $2$ satisfying  all of the
following conditions:
  \begin{enumerate}
  \item $\Sigma_\scrG$ is unramified at each of the following places:
    \begin{enumerate}[label=(\theenumi\alph*)]
    \item an infinite place of $F$ that is split in $A$, 
    \item a finite prime $\grp$ with $e_\grp(\calO)=2$, i.e.~$\calO_\grp\simeq
      \Mat_2(O_{F_\grp})$,
    \item a finite prime $\grp$ with $e_\grp(\calO)= 1$ and
      $n_\grp(\calO)\equiv 0\pmod{2}$,
    \item   a finite prime $\grp$ with $e_\grp(\calO)= -1$ and
     $A_\grp\simeq \Mat_2(F_\grp)$;
    \end{enumerate}
\item  $\Sigma_\scrG$ splits completely at each of the following
  finite prime $\grp$ of $F$: 
    \begin{enumerate}[label=(\theenumi\alph*)]
\item   $e_\grp(\calO)=-1$ and $A$ is ramified at $\grp$,
    \item $e_\grp(\calO)=1$ and
  $n_\grp(\calO)\equiv 1\pmod{2}$,
    \item  $e_\grp(\calO)=0$ and $n_\grp(\calO)=2$,
    \item  $e_\grp(\calO)=0$, $n_\grp(\calO)\geq 3$, $A$ is split at
      $\grp$, and $-1\not\in \grk_\grp^{\times 2}$,
          \item  $e_\grp(\calO)=0$, $n_\grp(\calO)\geq 3$, $A$ is ramified at
      $\grp$, and $-1\in \grk_\grp^{\times 2}$;
    \end{enumerate}
  \item if $\grp$ is a finite nondyadic
    prime of $F$ with $e_\grp(\calO)=0$, $n_\grp(\calO)\geq 3$, and 
    \begin{itemize}
    \item either $A$ is split at
      $\grp$ with $-1\in \grk_\grp^{\times 2}$,
      \item or $A$ is ramified  at
      $\grp$ with $-1\not\in \grk_\grp^{\times 2}$,  
    \end{itemize}
    then either $\grp$ splits completely in $\Sigma_\scrG$,  or $\Sigma_\scrG\otimes_F F_\grp$ is a direct
    sum of copies of a quadratic extension
    $M_\grp/F_\grp$ whose ring of integers  embeds into $\calO_\grp$.  
  \end{enumerate}
\end{prop}
Such a quadratic extension $M_\grp/F_\grp$ is necessarily ramified by
\cite[Proposition~1.12]{Brzezinski-crelle-1990}, and it is uniquely
determined by $\calO_\grp$ according to Lemma~\ref{lem:unique-ram-ext}. From  
Proposition~\ref{prop:description-Sigma}(2c), if $\grq$ is a dyadic prime with
$e_\grq(\calO)=0$ (hence $n_\grp(\calO)=2$ by assumption (\ref{eq:4})), then $\Sigma_\scrG/F$ splits completely at
$\grq$.  Note that if there exists a finite prime $\grp$ such that $A$
is ramified at $\grp$ and $\Sigma_\scrG$ splits completely at $\grp$,
then $B$ is \emph{not} optimally spinor selective for $\scrG$. Indeed, since $K_\grp$ embeds into
$A_\grp$ by our assumption, we must have $K\not\subseteq
\Sigma_\scrG$.  Proposition~\ref{prop:description-Sigma} will be proved at the end of
\S\ref{sec:bass-orders}.

Part (II) of Theorem~\ref{thm:main} follows directly from
\cite[Theorem~2.11]{xue-yu:spinor-class-no}. To prove the first part, we reduce it to local considerations as well. See
\S\ref{subsec:step-i} for  details.  It is clear from the above
discussion that we rely heavily on the fundamental work of
Brzezinski \cite{Brzezinski-crelle-1990}. In theory, it is possible to
apply his result to  remove assumption (\ref{eq:4}) and to obtain an
optimal spinor selectivity theorem for all quaternion Bass orders. However, our
method is built upon explicit computations, which becomes too
complicated at the dyadic primes. We leave such an endeavor to a
more adventurous reader.

\section{Quaternion Bass orders}  
\label{sec:bass-orders}
In this section, we recall the definition and basic properties of quaternion Bass
orders. Our main references are the work by Brzezinski
\cite{Brzezinski-crelle-1990,Brzezinski-1983} and by Chari et
al.~\cite{Voight-basic-orders}. We keep the notation of previous
sections, except that $F$ is allowed to be either a number field or a nonarchimedean local field of
characteristic not equal to two.

In the local case, a
quadratic extension of $F$  means a
quadratic semisimple $F$-algebra, that is, either $F\times F$ or a quadratic field
extension of $F$. We
denote  the unique maximal ideal
of $O_F$ by $\grp$, and  its residue field by $\grk$.  We drop the
subscript $_\grp$ and write $\nu, n(\calO)$, $i(B)$, $e(\calO)$ for 
$\nu_\grp, n_\grp(\calO)$, $i_\grp(B)$, $e_\grp(\calO)$,  respectively. See
(\ref{eq:3}) and Definition~\ref{defn:eichler-invariant}. 

Given an order $\calO$ in the
quaternion $F$-algebra $A$, an \emph{overorder}  of $\calO$ is an order $\calO'$
in $A$  containing $\calO$. An overorder $\calO'\supsetneq \calO$ 
 is called a \emph{minimal overorder} of $\calO$ if it is minimal with respective to inclusion
 among the orders  \emph{properly} containing $\calO$.


\begin{defn}
  An order $\calO$ in $A$ is \emph{Gorenstein} if its dual lattice
  $\calO^\vee:=\{x\in A\mid \Tr(x\calO)\subseteq O_F\}$ is 
  projective as a left (or right) $\calO$-module. It is
  called a \emph{Bass order} if every orverorder of $\calO$ (including
  $\calO$ itself)  is Gorenstein. 
\end{defn}

As noted by Bass \cite{Bass-MathZ-1963} himself, Gorenstein orders are
ubiquitous.  Being Gorenstein is a local property (when $F$ is a number
field), that is, $\calO$ is
Gorenstein if and only if  $\calO_\grp$ is Gorenstein for every finite
prime $\grp$ of $F$. Consequently, being 
Bass is  a local property as well. Bass orders enjoy many equivalent
characterizations.  We merely mention one of them that is mostly
relevant to our current quest. 
\begin{thm}
  An order $\calO\subset A$ is Bass if and only if it is \emph{basic}, i.e.~there exists a semisimple quadratic $F$-algebra $L$ whose ring of
  integers $O_L$  embeds into $\calO$. 
\end{thm}

This theorem is proved  by
Brzezinski \cite[Proposition~1.11]{Brzezinski-crelle-1990} in the local
case, and by Chari et al. \cite[Theorem~1.2]{Voight-basic-orders} in
the number field case. See \cite[Corollary~1.3]{Voight-basic-orders}
for more characterizations of quaternion Bass orders.

For the rest of this section, we assume that $F$ is local unless
specified otherwise.  Recall
that an order $\calO$ is \emph{hereditary} if 
 every  $\calO$-lattice in a free $A$-module is
$\calO$-projective \cite[p.~76]{curtis-reiner:1}. 
From \cite[Proposition~1.2]{Brzezinski-1983},  $\calO$ is hereditary
if and only if $n(\calO)\leq 1$. If $\calO$ is Bass but
non-hereditary, then we have some further information on the quadratic
$F$-algebra $L$ with $\Emb(O_L, \calO)\neq \emptyset$ from 
\cite[Proposition~1.12]{Brzezinski-crelle-1990}:
\begin{itemize}
\item if $e(\calO)=1$, then  $L=F\times F$;
\item if $e(\calO)=-1$, then $L/F$ is the
  unique quadratic unramified field extension;
\item if $e(\calO)=0$, then $L/F$ is a ramified field extension.   
\end{itemize}
In fact, if $e(\calO)=0$ and $n(\calO)=2$, then $L/F$ can be 
any arbitrary quadratic ramified extension according to
\cite[(3.14)]{Brzezinski-crelle-1990}. If $F$ is nondyadic,
$e(\calO)=0$ and $n(\calO)\geq 3$, then we prove in
Lemma~\ref{lem:unique-ram-ext} that such an  $L/F$
is uniquely determined by $\calO$.

As
mentioned in Remark~\ref{rem:special-cases}, any order $\calO$ with
$e(\calO)\neq 0$ is automatically Bass.  If $e(\calO)=1$, then $\calO$
is a non-maximal  Eichler order, and it has exactly two minimal
overorders. Suppose that 
$e(\calO)\in \{-1, 0\}$ and $\calO$ is Bass but non-hereditary. Then from \cite[Proposition~1.12]{Brzezinski-1983}, $\calO$
has a unique minimal overorder $\calM(\calO)$, which is Bass by
default.  From \cite[Propositions~3.1 and
4.1]{Brzezinski-1983},
\begin{equation}
  \label{eq:40}
  n(\calM(\calO))=
  \begin{cases}
    n(\calO)-2 &\text{if } e(\calO)=-1,\\
    n(\calO)-1 &\text{if } e(\calO)=0,
  \end{cases}
\end{equation}
and $e(\calM(\calO))=e(\calO)$ if 
 $\calM(\calO)$ is again non-hereditary. 
 Thus starting from
$\calM^0(\calO):=\calO$, we  define $\calM^i(\calO):=\calM(\calM^{i-1}(\calO))$ recursively
to obtain a unique chain of Bass orders  terminating at a hereditary order
$\calM^m(\calO)$:
\begin{equation}
  \label{eq:11}
 \calO=\calM^0(\calO)\subset \calM^1(\calO)\subset
  \calM^2(\calO)\subset \cdots \subset \calM^{m-1}(\calO)\subset \calM^m(\calO), 
\end{equation}
where each $\calM^i(\calO)$ is a Bass
non-hereditary order for  $0\leq i\leq
m-1$. Furthermore, 
\begin{itemize}
\item $m=n(\calO)-1$ if $e(\calO)=0$; and 
\item $m=\fl{n(\calO)/2}$ if $e(\calO)=-1$, where $x\mapsto \fl{x}$ is the
  floor function.
\end{itemize}
The order $\calM^m(\calO)$ is called the \emph{hereditary closure} of
$\calO$ and will henceforth be denoted by $\calH(\calO)$. If $e(\calO)=-1$, then
$\calH(\calO)$ is always a maximal order by
\cite[Proposition~3.1]{Brzezinski-1983}. Thus when $e(\calO)=-1$,
$n(\calO)$ is even if $A\simeq \Mat_2(F)$, and $n(\calO)$ is odd if
$A$ is ramified (i.e.~$A$ is division).

Recall that $\nu: F^\times\twoheadrightarrow \Z$ denotes the discrete
valuation of $F$. We say that an element $x\in A^\times$ is \emph{even
  (resp.~odd)} if $\nu(\Nr(x))$ is even (resp.~odd). Let
$\calN^0(\calO)$ be the \emph{even normalizer group} of $\calO$, that is,  
\begin{equation}
  \label{eq:35}
  \calN^0(\calO):=\{x\in A^\times\mid x\calO x^{-1}=\calO, \text{ and
  } \nu(\Nr(x))\equiv 0\pmod{2}\}. 
\end{equation}

\begin{lem}
  If $e(\calO)=-1$, then
  \begin{equation}
    \label{eq:21}
\Nr(\calN(\calO))=
  \begin{cases}
    F^\times &\text{if $A$ is ramified}, \\
    F^{\times2}O_F^\times &\text{if $A$ is split.} 
  \end{cases}
  \end{equation}
\end{lem}
\begin{proof}
  Let $E:=E_\ur$ be the unique unramified quadratic field extension of $F$.  Since
  $e(\calO)=-1$,  there
  exists an embedding $\varphi: O_E\to \calO$. Indeed, if $\calO$ is
  nonhereditary, this follows from
  \cite[Proposition~1.12]{Brzezinski-crelle-1990} as discussed 
   above. If $\calO$ is
  hereditary, then  $n(\calO)=1$, which implies that $A$ is division
  and  $\calO$ is the unique maximal
  order in $A$.  Hence $\Emb(O_E, \calO)\neq
  \emptyset$ by \cite[Corollary~II.1.7]{vigneras}. For simplicity, let
  us identify $O_E$ with its image in $\calO$ via $\varphi$. From
  \cite[Proposition~V.1]{Serre_local},  $\Nm_{E/F}(O_E^\times)=O_F^\times$. It follows
  that 
  \[\Nr(\calN(\calO))\supseteq \Nr(F^\times O_E^\times)=F^{\times
      2}O_F^\times, \]
  which is an index $2$-subgroup in $F^\times$. 
On the other hand, from 
\cite[Theorem~2.2]{Brzezinski-crelle-1990}, $\calN(\calO)=\calN^0(\calO)$ if and only if $n(\calO)$
is even (which happens if and only if $A$ is split as observed as above). The lemma follows immediately. 
\end{proof}

In the local case, Bass orders  can be described explicitly according
to\footnote{Caution: there is a minor typo in the line immediately
  below equation (1.5) in \cite{Brzezinski-crelle-1990}, instead of
  $0\leq r-s\leq 1$, it should read $0\leq s-r\leq 1$. Compare with
  \cite[Proposition~5.4(c)]{Brzezinski-1983}. We should emphasize that
  this bears no effect on the validity of results of
  \cite{Brzezinski-crelle-1990}, as most of the deduction relies on
  \cite[(2.8)]{Brzezinski-crelle-1990} instead.}
\cite[\S1]{Brzezinski-crelle-1990} (see also \cite[Propositions~5.4
and 5.6]{Brzezinski-1983}). For reasons to be explained 
in the proof of Theorem~\ref{thm:main} (see \S\ref{subsec:step-i}), we focus exclusively on
local Bass orders of Eichler invariant zero.


Let $\pi$ be a uniformizer of $F$. Pick an element $\varepsilon\in F$
in the following way:
\begin{itemize}
\item if $A$ is split (i.e.~$A\simeq \Mat_2(F)$), then put
  $\varepsilon=0$;
\item if $A$ is ramified, then choose $\varepsilon\in O_F^\times$ such
  that 
  $1-4\varepsilon\in O_F^\times\smallsetminus O_F^{\times 2}$
  (i.e.~$1-4\varepsilon$ is a unit but a non-square). The existence of
  such a unit is guaranteed by \cite[63:4]{o-meara-quad-forms}. 
\end{itemize}
In the ramified case, the assumption on $\varepsilon$ implies that
\begin{equation}
  \label{eq:20}
1+\beta+\varepsilon\beta^2\in O_F^\times, \qquad \forall \beta\in
  O_F. 
\end{equation}

We can  choose an $F$-basis $\{1, x_1, x_2, x_3\}$ of $A$ satisfying the
following conditions:
\begin{equation}
  \label{eq:16}
  x_1^2=x_1-\varepsilon, \quad x_2^2=\pi,\quad x_2x_1=(1-x_1)x_2,
  \quad x_3=x_1x_2.
\end{equation}
Indeed, if $A$ is split, then we put as in \cite[(2.7)]{Brzezinski-crelle-1990}:
\begin{equation}
  \label{eq:17}
  x_1:=
  \begin{bmatrix}
 1 & 0 \\ 0 & 0    
  \end{bmatrix},\qquad   x_2:=\begin{bmatrix}
 0 & 1 \\ \pi & 0    
  \end{bmatrix},\qquad x_3:= \begin{bmatrix}
 0 & 1 \\ 0 & 0    
  \end{bmatrix}.
\end{equation}
If $A$ is ramified, then from \cite[63:3]{o-meara-quad-forms},
$F(x_1)$ coincides with $E_\ur$, the unique unramified quadratic  field 
extension of $F$, and the
existence of  a basis satisfying (\ref{eq:16}) is guaranteed by
\cite[Corollary~II.1.7]{vigneras}. 

Given a set $X$ in a finite dimensional $F$-vector space $V$, we write $\dangle{X}$ for the
$O_F$-submodule of $V$ spanned by $X$. As remarked right after
Definition~\ref{defn:eichler-invariant}, if $e(\calO)=0$, then
$n(\calO)\geq 2$. 
According to
\cite[(2.8)]{Brzezinski-crelle-1990}, after replacing $\pi$ by another
suitable uniformizer if necessary, every Bass order $\calO$ with
$e(\calO)=0$ and $n(\calO)=n$ 
 is isomorphic to
\begin{gather}
  \label{eq:18}
  \dangle{1,\ x_{\alpha \beta},\ \pi^r x_1,\ \pi^s x_3}, \qquad
  \text{where}\\
r+s=n-1,\quad 0\leq r-s\leq 1,\quad x_{\alpha\beta}=\alpha
x_1+x_2+\beta x_3, \quad \alpha\in \grp,\ \beta\in O_F, \label{eq:19}\\
\text{and}\qquad 1+\beta\in O_F^\times \quad\text{if $A$ is split} \label{eq:27}. 
\end{gather}
There is no other restriction on $\beta$  when
$A$ is ramified. From (\ref{eq:19}),
\begin{equation}
  \label{eq:33}
  r=\fl{\frac{n}{2}}\geq 1, \quad \text{and} \quad
  s=\fl{\frac{n-1}{2}}\geq 0.  
\end{equation}

Let $\calO$ be as in (\ref{eq:18}). Given $y\in A$, the discriminant  of $y$ is defined to be
$\Delta(y):=\Tr(y)^2-4\Nr(y)$. If $y\in \calO$, we write
$y$ uniquely as $a+bx_{\alpha\beta}+c\pi^r x_1+d\pi^sx_3$ with $a, b,
c, d\in O_F$. Direct calculation yields (see \cite[(3.19), (3.20)]{Brzezinski-crelle-1990}) 
\begin{align}
\Tr(y)=&2a+b\alpha+c\pi^r,\label{eq:29}\\
\begin{split} \Nr(y)=&a^2+a(b \alpha+c
    \pi^r)+\varepsilon(b \alpha+c \pi^r)^2 \\
  &-\pi\left[b^2+b(b \beta+d \pi^s)+\varepsilon(b
      \beta+d \pi^s)^2\right], \end{split}\label{eq:26}\\
  \begin{split} \Delta(y)=&(b \alpha+c \pi^r)^2(1-4 \varepsilon) \\ &+4\pi\left[b^2+b(b \beta+d \pi^s)+\varepsilon(b \beta+d \pi^s)^2\right]. \end{split} \label{eq:24}
\end{align}
In particular, we have 
\begin{equation*}
  \Tr(x_{\alpha\beta})=\alpha, \quad
  \Nr(x_{\alpha\beta})=\alpha^2\varepsilon-\pi(1+\beta+\varepsilon\beta^2),\quad
  \Delta(x_{\alpha\beta})=\alpha^2(1-4\varepsilon)+4\pi(1+\beta+\varepsilon\beta^2). 
\end{equation*}
Since $\pi\vert \alpha$ by
(\ref{eq:19}) and $1+\beta+\varepsilon \beta^2\in O_F^\times$ by
(\ref{eq:20}) and (\ref{eq:27}), $F(x_{\alpha\beta})$ is a ramified quadratic extension of
$F$,  and $x_{\alpha\beta}$ is a uniformizer of $F(x_{\alpha\beta})$. 
Particularly, the  ring of integers of $F(x_{\alpha\beta})$ coincides with $\dangle{1,
  x_{\alpha\beta}}$.
 Since $r\geq 1$ by (\ref{eq:33}),  it is clear
from  (\ref{eq:26}) that  
\begin{equation}
  \label{eq:42}
  y\in \calO^\times\quad \text{if and only if}\quad a\in O_F^\times. 
\end{equation}

\begin{lem}
  If $e(\calO)=0$ and $F$ is nondyadic, then
  \begin{equation}
  \label{eq:43}
  \Nr(\calO^\times)=O_F^{\times2}.
\end{equation}
\end{lem}
\begin{proof}
For any $u\in O_F^\times$,
we write $\tilde{u}$ for its canonical image in the residue field
$\grk=O_F/\grp$. From Hensel's lemma and the assumption that
$\fchar(\grk)\neq 2$, we have 
\begin{equation}
  \label{eq:34}
O_F^{\times 2}=\{u\in O_F^\times\mid \tilde{u}\in
\grk^{\times 2}\}. 
\end{equation}
The equality (\ref{eq:43}) follows directly by combining
(\ref{eq:26}), (\ref{eq:42})
and (\ref{eq:34}). 
\end{proof}

When $F$ is nondyadic,  it  
 has exactly two ramified quadratic
extensions up to isomorphism, namely $F(\sqrt{\pi})$ and $F(\sqrt{\pi u})$, where $u\in
O_F^\times\smallsetminus O_F^{\times 2}$.
\begin{lem}\label{lem:unique-ram-ext}
Suppose that $F$ is nondyadic, and $\calO\subset A$ is a Bass order
with $e(\calO)=0$ and $n(\calO)\geq 3$. Up to isomorphism,  there exists a \emph{unique}
 quadratic extension $M/F$ such that $O_M$ embeds into
$\calO$.
\end{lem}
\begin{proof}
  Only the uniqueness of such an $M$ need to be proved. From
  \cite[Proposition~1.12]{Brzezinski-crelle-1990}, $M/F$ is
  necessarily ramified. 
  Write $\calO$  as in
  (\ref{eq:18}) and suppose that  $M/F$ is a ramified quadratic
extension such that there exists an embedding $\varphi:
  O_M\to \calO$. According to \cite[Proposition~I.18]{Serre_local},
  the characteristic polynomial over $F$ of the uniformizer $\pi_M\in M$  is
  an Eisenstein polynomial of degree $2$, and $O_M=\dangle{1, \pi_M}$.
  If we put $y:=\varphi(\pi_M)$, then $\Tr(y)\in \grp$ and $\Nr(y)\in
  O_F^\times\pi$. Write $y=a+bx_{\alpha\beta}+c\pi^r x_1+d\pi^sx_3$ with $a, b,
  c, d\in O_F$ as before. Since $n(\calO)\geq 3$, we have $r\geq s\geq
  1$ by (\ref{eq:33}). 
  From (\ref{eq:29}) and (\ref{eq:26}), the
previous conditions on $\Tr(y)$ and $\Nr(y)$ are equivalent to 
\begin{equation}
  \label{eq:32}
  \pi\vert a \quad\text{and}\quad b\in O_F^\times. 
\end{equation}
Indeed, since $F$ is nondyadic and $\pi|\alpha$ by (\ref{eq:19}), we
have 
$\Tr(y)\in \grp$ if and only if $\pi\vert a$ by (\ref{eq:29}). Suppose
further that $\pi \vert
a$. Then from (\ref{eq:26}), we have
\begin{equation}
  \label{eq:28}
  \begin{split}
    \Nr(y)\equiv -\pi b^2(1+\beta+\varepsilon\beta^2) \pmod{\pi^2}, 
  \end{split}
\end{equation}
which implies that
\[  \Nr(y)\in
  O_F^\times\pi\quad \text{if and only if}\quad  b^2(1+\beta+\beta^2\varepsilon)\in O_F^\times.\]
 From
(\ref{eq:20}) and (\ref{eq:27}), $1+\beta+\varepsilon\beta^2\in O_F^\times$ by our choices of
$\varepsilon$ and $\beta$.  Thus $b\in O_F^\times$ is both necessary  and
sufficient.

Now an easy calculation shows that 
$\Delta(y)/\Delta(x_{\alpha\beta})\equiv b^2\pmod{\grp}$.  It follows
from (\ref{eq:34}) that $\Delta(y)\in
\Delta(x_{\alpha\beta})O_F^{\times2}$. Therefore, $M\simeq
F(x_{\alpha\beta})$, and hence it is uniquely determined  up to isomorphism.  
\end{proof}

Let $\calO$ be an arbitrary Bass order of Eichler invariant
  $0$.  The normalizer group $\calN(\calO)$ has been described effectively in
\cite[Theorem~2.5]{Brzezinski-crelle-1990}. First, suppose that
$n(\calO)=2$. Then $\calM(\calO)$ is the   hereditary
closure of $\calO$, and 
\begin{equation}
  \label{eq:59}
\calO=O_F+\grJ(\calM(\calO)),  
\end{equation}
 where $\grJ(\calM(\calO))$ denotes the
Jacobson radical of  $\calM(\calO)$. It follows that $ \calN(\calO)=\calN(\calM(\calO))$.
If $A$ is split, then $\calM(\calO)$ is an Eichler order of level
$\grp$; if $A$ is ramified, then $\calM(\calO)$ is the unique maximal
order. In both cases, the reduced norm of the normalizer group $\calN(\calM(\calO))$
coincides with $F^\times$. Therefore,
\begin{equation}
  \label{eq:48}
  \Nr(\calN(\calO))=F^\times \qquad\text{if $e(\calO)=0$ and
    $n(\calO)=2$}. 
\end{equation}
Next, suppose that $n(\calO)\geq 3$. Write $\calO$ as in (\ref{eq:18})
and put $M=F(x_{\alpha\beta})$. 
 From  \cite[Theorem~2.5]{Brzezinski-crelle-1990},  
the even normalizer group $\calN^0(\calO)$ is a  subgroup of index $2$ in  $\calN(\calO)$,
and  
\begin{equation}
  \label{eq:36}
  \calN(\calO)=\calN^0(\calO)\bigsqcup \calN^0(\calO)x_{\alpha\beta}. 
\end{equation}
See \cite[p.~177]{Brzezinski-crelle-1990}. In particular,
\begin{equation}
  \label{eq:37}
\calN(\calO)\supseteq x_{\alpha\beta}^\Z\cdot O_M^\times=M^\times. 
\end{equation}
Since $[F^\times:\Nm_{M/F}(M^\times)]=2$ by local class field theory,
we find that
\begin{equation}
  \label{eq:46}
\Nr(\calN(\calO))\text{ coincides with either } F^\times \text{ or } \Nm_{M/F}(M^\times).
\end{equation}
For simplicity, we only write down $\calN^0(\calO)$ under the assumption that $F$ is nondyadic:
\begin{equation}
  \label{eq:39}
  \calN^0(\calO)=    F^\times\calM(\calO)^\times\bigsqcup
  F^\times\calM(\calO)^\times\sigma_0 \qquad \text{if } n(\calO)\geq 3,
\end{equation}
where $\sigma_0\in \calH(\calO)^\times\smallsetminus
\calM^{n-2}(\calO)^\times$. More explicitly, from
\cite[(2.10)]{Brzezinski-crelle-1990}, we have 
\begin{align}
  \label{eq:41}
  \sigma_0&=-1-\frac{A\alpha}{2}+Ax_{\alpha\beta}+2x_1,  \qquad
  \text{where}\\
  A&=\frac{-2\alpha(1-4\varepsilon)}{\alpha^2(1-4\varepsilon)+4\pi(1+\beta+\beta^2\varepsilon)}\in
  O_F. \label{eq:44}
\end{align}
By direct calculation,
\begin{equation}
  \label{eq:45}
  \Nr(\sigma_0)=-\left(1+\frac{A\alpha}{2}\right)^2(1-4\varepsilon)-\pi
  A^2(1+\beta+\beta^2\varepsilon)\in
  O_F^\times. 
\end{equation}

\begin{lem}\label{lem:normalizer-red-norm}
Write $\calO$ as in (\ref{eq:18}) and put $M=F(x_{\alpha\beta})$. Suppose
that $n(\calO)\geq 3$ and $F$ is nondyadic.  Then
  $  \Nr(\calN(\calO))=\Nm_{M/F}(M^\times) $ if and only if one of the
  following conditions holds:
  \begin{enumerate}[label=(\roman*)]
  \item $A$ is split, and $-1\in \grk^{\times 2}$;
  \item   $A$ is ramified, and $-1\not\in \grk^{\times 2}$. 
  \end{enumerate}
\end{lem}
\begin{proof}
Since $M/F$ is tamely ramified, $\Nm_{M/F}(O_M^\times)=O_F^{\times 2}$
according to \cite[Corollary~V.7]{Serre_local}. Combining
(\ref{eq:43}), (\ref{eq:36}) and (\ref{eq:39}), we see that
$  \Nr(\calN(\calO))=\Nm_{M/F}(M^\times) $ if and only if
$\Nr(\sigma_0)\in O_F^{\times 2}$. If $A$ is split, then
$\varepsilon=0$ by construction. It follows from (\ref{eq:34}) that
$\Nr(\sigma_0)\in O_F^{\times 2}$ if and only if $-1\in F^{\times
  2}$ in this case. Next, suppose that $A$ is ramified.  Then
$(1-4\varepsilon)\not\in O_F^{\times 2}$ by construction. Thus $\Nr(\sigma_0)\in O_F^{\times 2}$ if and only if $-1\not\in F^{\times
  2}$ in this case. Lastly, since $F$ is nondyadic, $-1\in F^{\times
  2}$ if and only if $-1\in \grk^{\times 2}$. 
\end{proof}

\begin{proof}[Proof of Proposition~\ref{prop:description-Sigma}]
  Suppose that $F$ is a number field, and $\scrG$ is a genus of Bass
  orders in $A$ satisfying (\ref{eq:4}). Let $\calO$ be an arbitrary member of
  $\scrG$,  and $\Sigma_\scrG$ be the spinor
  genus field of $\scrG$.    The description of $\Sigma_\scrG$ at the
  infinite places  of $F$ or at a finite prime $\grp$ with
  $e_\grp(\calO)\in \{1, 2\}$ is well known. See
  \cite[Proposition~31.2.1]{voight-quat-book} for example. Thus we focus on the
  finite primes $\grp$ with $e_\grp(\calO)\in \{ -1, 0\}$. 

  First, suppose that   $e_\grp(\calO)=-1$.  If $A_\grp$ is split, then
  $\Nr(\calN(\calO_\grp))=F_\grp^{\times2}O_{F_\grp}^\times$ by (\ref{eq:21}). Hence
  $\Sigma_\scrG/F$ is unramified at $\grp$. This
  proves part (1d) of
  Proposition~\ref{prop:description-Sigma}. Similarly, if $A_\grp$ is
  ramified, then $\Nr(\calN(\calO_\grp))=F_\grp^\times$, and hence
  $\Sigma_\scrG/F$ splits completely at $\grp$. Part (2a) of the
  proposition follows.

  Next, suppose that $e_\grp(\calO)=0$. 
  If one of the following
  conditions holds:
  \begin{itemize}
  \item $n_\grp(\calO)=2$,
      \item  $e_\grp(\calO)=0$, $n_\grp(\calO)\geq 3$, $A$ is split at
      $\grp$, and $-1\not\in \grk^{\times 2}$,
          \item  $e_\grp(\calO)=0$, $n_\grp(\calO)\geq 3$, $A$ is ramified at
      $\grp$, and $-1\in \grk^{\times 2}$,  
   \end{itemize}
then $\Nr(\calN(\calO_\grp))=F_\grp^\times$ according to
(\ref{eq:48}), (\ref{eq:46})  and
Lemma~\ref{lem:normalizer-red-norm}. Note that the condition
  $n_\grp(\calO)\geq 3$ implies that $\grp$ is nondyadic by assumption
  (\ref{eq:4}).  This proves part (2c)--(2e) of
Proposition~\ref{prop:description-Sigma}. Part (3) of the
proposition follows directly from
Lemma~\ref{lem:normalizer-red-norm}. 
\end{proof}

%


\section{The proof of the main theorem}
\label{sec:proof-main}
We carry out the proof of Theorem~\ref{thm:main}
 in  three  steps. 
\numberwithin{thmcounter}{subsection}
\subsection{Step (I): reduction to the local case}
\label{subsec:step-i}
Let $F$ be a number
field, and  $\scrG$ be a genus of Bass
  orders in $A$ satisfying  condition (\ref{eq:4}). Fix an order
  $\calO$ in $\scrG$.  Let $K/F$ be a quadratic field
  extension embeddable into $A$, and $B$ be an order in $K$.  Suppose
  that  for each finite prime $\grp$ of $F$, there exists an optimal
  embedding $\varphi_\grp: B_\grp\to \calO_\grp$. Put
\begin{equation}
  \label{eq:6}
  \calE_\grp(\varphi_\grp, B_\grp, \calO_\grp):=\{g_\grp\in A_\grp^\times \mid \varphi_\grp(K_\grp)\cap g_\grp\calO_\grp
  g_\grp^{-1}=\varphi_\grp(B_\grp)\}. 
\end{equation}
If $\varphi_\grp, B_\grp$ and $\calO_\grp$ are clear
from the context, we simply write $\calE_\grp$ for $\calE_\grp(\varphi_\grp, B_\grp, \calO_\grp)$. 
Clearly, $\calE_\grp$ is invariant under left translation by 
$\varphi(K_\grp^\times)$ and right translation by
$\calN(\calO_\grp)$.  The map $g_\grp\mapsto g_\grp^{-1}\varphi_\grp g_\grp$ induces a
bijection 
\begin{equation}
  \label{eq:70}
   \varphi_\grp(K_\grp^\times)\bsh \calE_\grp\simeq \Emb(B_\grp, \calO_\grp).
\end{equation} Moreover, 
\begin{equation}
  \label{eq:54}
\calE_\grp\supseteq \varphi(K_\grp^\times)\calN(\calO_\grp).
\end{equation}
It follows that $\Nr(\calE_\grp)$ is a subgroup of
$F_\grp^\times$ of index at most $2$, and  $\Nr(\calE_\grp)$ 
does not depend on the choices of $\varphi_\grp\in \Emb(B_\grp, \calO_\grp)$.  See \cite[\S2.2]{xue-yu:spinor-class-no}.
Let us define a finite set of  primes of $F$ as follows
\begin{equation}
  \label{eq:23}
  S:= \{\grp\mid e_\grp(\calO)=0\quad\text{and}\quad (K/\grp)\neq 1\}.
\end{equation}
From \cite[Theorem~2.11]{xue-yu:spinor-class-no}, $B$
is optimally spinor selective for $\scrG$ if and only if
\begin{equation}
  \label{eq:7}
  K\subseteq \Sigma_\scrG, \quad \text{and}\quad
  \Nm_{K/F}(K_\grp^\times)=\Nr(\calE_\grp)\text{ for every } \grp\in S. 
\end{equation}

Assume that $K\subseteq \Sigma_\scrG$ from now on.  From
Proposition~\ref{prop:description-Sigma}, if $\grp\in S$, then both of
the following conditions hold:
\begin{enumerate}
\item[($\dagger$)] $\grp$ is nondyadic, $e_\grp(\calO)=0$, and   $n_\grp(\calO)\geq 3$,
\item[($\ddagger$)]  $K_\grp$ is the unique ramified quadratic extension of $F_\grp$
  such that $O_{K_\grp}$ embeds into $\calO_\grp$, and
  $\Nr(\calN(\calO_\grp))=\Nm_{K_\grp/F_\grp}(K_\grp^\times)$. From
  Lemma~\ref{lem:normalizer-red-norm}, the
  last identity is equivalent to the following:
  \begin{itemize}
    \item either $A_\grp$ is split with $-1\in \grk_\grp^{\times 2}$,
      \item or $A_\grp$ is ramified with $-1\not\in \grk_\grp^{\times 2}$.  
    \end{itemize}
\end{enumerate}
Thus we may
replace $S$ in (\ref{eq:23}) by
\begin{equation}
  \label{eq:50}
S:= \{\grp\mid  \text{both conditions ($\dagger$) and
  ($\ddagger$) hold at } \grp\}.
\end{equation}
To prove Theorem~\ref{thm:main},  we need to show that for each
$\grp\in S$ as above,    $\Nr(K_\grp^\times)=\Nr(\calE_\grp)$ 
if and only if one of the following conditions holds
\begin{enumerate}[label=(\roman*)]
   \item  $n_\grp(\calO)\geq 2i_\grp(B)+3$, 
   \item $n_\grp(\calO)=2i_\grp(B)+1$, $A$ is split at $\grp$, and $\abs{\grk_\grp}=5$,
   \item   $n_\grp(\calO)=2i_\grp(B)+1$, $A$ is ramified at $\grp$, and $\abs{\grk_\grp}=3$. 
   \end{enumerate}
   See (\ref{eq:3}) for the definitions of $n_\grp(\calO)$ and
   $i_\grp(B)$. If $A$ is split at $\grp$, then the assumption that $\grp$ is nondyadic and
$-1\in \grk_\grp^{\times2}$ already implies that $\abs{\grk_\grp}\geq
5$.  Thus in (ii) or (iii), the value of $\abs{\grk_\grp}$ is
precisely the minimal one  in each respective case. 


\subsection{Step (II): the recursion}
\label{subsec:step-ii}
From now on, we work exclusively in the local case under the assumptions
($\dagger$) and ($\ddagger$) above. More explicitly,  $F$ is
assumed to be a nonarchimedean nondyadic local field with prime ideal $\grp$ and
residue field $\grk$, and $\calO$ is  a Bass
order  in $A$ with $e(\calO)=0$ and $n(\calO)\geq 3$.  Moreover,  $K/F$ is the unique
ramified quadratic extension such that $O_K$ embeds into $\calO$, and
$\Nr(\calN(\calO))=\Nm_{K/F}(K^\times)$. Let $B$ be an order in $K$
with $\Emb(B, \calO)\neq \emptyset$. 
We drop the subscript $_\grp$ and write $\varphi, \calE$ for
$\varphi_\grp,
\calE_\grp$ etc. 

\begin{lem}\label{lem:case-B=OK}
 If $B=O_K$,  then $\Nr(\calE)=\Nm_{K/F}(K^\times)$. 
\end{lem}
\begin{proof}
  From the proof of \cite[Theorem~3.10,
  p.~180]{Brzezinski-crelle-1990}, the even normalizer group
  $\calN^0(\calO)$ acts \emph{transitively} from the right by conjugation on the set of (optimal)
  embeddings $\Emb(O_K, \calO)$. Hence we have
  \begin{equation}
    \label{eq:53}
    \calE=\varphi(K^\times)\calN^0(\calO). 
  \end{equation}
  Since $\Nr(\calN(\calO))= \Nm_{K/F}(K^\times)$ by our
  assumption, the equality $\Nr(\calE)=\Nm_{K/F}(K^\times)$ follows
  directly from (\ref{eq:53}). 
\end{proof}

Now assume that $i(B)\geq 1$. For simplicity, let $\calM(B)$ be the unique
order in $K$ such that $i(\calM(B))=i(B)-1$. A corner stone of our
proof of Theorem~\ref{thm:main} is the following (slightly adjusted)
lemma of Brzezinski \cite[Lemma~3.18]{Brzezinski-crelle-1990}.
\begin{lem}\label{lem:recursive}
  Let $L/F$ be a semisimple quadratic extension, and $R\subset O_L$
  be an order with  $i(R)\geq 1$. Then every optimal
  embedding $R\to \calO$ extends to an optimal embedding $\calM(R)\to
  \calM^2(\calO)$, and every optimal embedding $\calM(R)\to
  \calM^2(\calO)$ whose image is not in $\calO$ restricts to an optimal
  embedding $R\to \calO$. Moreover, for each optimal embedding $\calM(R)\to
  \calM^2(\calO)$, its image is not in $\calO$ with the exception of
  $L=K$ and  $\calM(R)=O_K$ (i.e.~$i(R)=1$).  
\end{lem}

For the exceptional case, Brzezinski  has ``$L\supset F$
ramified'' instead of ``$L=K$''. However, since we assume that $F$ is
nondyadic and $n(\calO)\geq 3$, if $L/F$ is ramified and $L\neq K$,
then $\Emb(O_L, \calO)=\emptyset$  by Lemma~\ref{lem:unique-ram-ext}.

Now fix $\varphi\in \Emb(B, \calO)$. Applying
Lemma~\ref{lem:recursive} for  $L=K$ and $R=B$, we see that 
\begin{equation}
  \label{eq:55}
\varphi\in \Emb(\calM(B), \calM^2(\calO))\quad\text{and}\quad  \calE(\varphi, B, \calO)\subseteq \calE(\varphi, \calM(B), \calM^2(\calO)).
\end{equation}
Moreover,
\begin{equation}
  \label{eq:56}
  \calE(\varphi, B, \calO)=\calE(\varphi, \calM(B), \calM^2(\calO)) \quad \text{if
  } i(B)>1.
\end{equation}
Starting from a pair  $(B, \calO)$ with $i(B)\geq 1$ and $n(\calO)\geq
3$, we apply (\ref{eq:56}) repeatedly until we arrive
at a pair of orders $(\wt{B}, \wt{\calO})$ for which (\ref{eq:56}) no
longer applies. In other words, the recursion
halts after $k$ steps once we hit one of the following conditions:
\begin{equation}
  \label{eq:57}
i(\wt{B})=i(B)-k=1 \quad\text{or}\quad n(\wt\calO)=n(\calO)-2k<3.
\end{equation}
For simplicity, let us put $\wt\calE:=\calE(\varphi, \wt{B},
\wt\calO)$. By construction, $\calE=\wt\calE$, so 
\begin{equation}
  \label{eq:58}
\Nr(\calE)=\Nm_{K/F}(K^\times) \quad \text{if and only if}\quad
\Nr(\wt\calE)=\Nm_{K/F}(K^\times). 
\end{equation}
Thus  we may replace $(B, \calO)$ by
$(\wt B, \wt \calO)$ and try to character when $\Nr(\wt\calE)=\Nm_{K/F}(K^\times)$ holds true.
Depending on the halting condition and the output of the recursion, the discussion will be separated
into the four cases according to the following table.
\begin{table}[htbp]
  \caption{the recursion}\label{tab:recursion}
\renewcommand{\arraystretch}{1.3}
 \begin{tabular}{*{3}{|>{$}c<{$}}|}
\hline
 \text{Start} & \text{Number of steps}& \text{Finish}                              \\
   \hline
   n(\calO)\leq 2i(B) &k=\fl{(n(\calO)-1)/2} & n(\wt\calO)\in\{1,2\},\ i(\wt B)\geq 1\\ 
   \hline
   n(\calO)=2i(B)+1 &k=i(B)-1  & n(\wt\calO)=3,\ i(\wt B)=1\\ \hline
   n(\calO)=2i(B)+2 &k=i(B)-1& n(\wt\calO)=4,\ i(\wt B)=1\\  \hline
   n(\calO)\geq 2i(B)+3&k=i(B)-1&n(\wt\calO)\geq 5,\ i(\wt B)=1 \\ 
   \hline
\end{tabular}  
\end{table}


\subsection{Step (III): the case by case study}
\label{subsec:step-iii}
Keep the notation and  assumptions of the previous step. 
\begin{lem}\label{lem:neq-first-case}
If $n(\calO)\leq 2i(B)$,  then $  \Nr(\calE)\neq \Nm_{K/F}(K^\times)$. 
\end{lem}
\begin{proof}
  In this case, we have $n(\wt\calO)\in \{1, 2\}$ and
  $i(\wt B)\geq 1$. If $n(\wt\calO)=2$, then
  $\wt{\calO}=O_F+\grJ(\calM(\wt\calO))$ by (\ref{eq:59}), which
  implies that  $\Emb(\wt{B}, \wt{\calO})=\Emb(\wt{B},
  \calM(\wt\calO))$ by   the discussion at the bottom of
  \cite[p.~181]{Brzezinski-crelle-1990}.  Thus $\varphi\in \Emb(\wt{B}, \calM(\wt\calO))$ and
  $\wt\calE=\calE(\varphi, \wt{B},
  \calM(\wt\calO))$. Replacing $\wt\calO$ by $\calM(\wt\calO)$ if
  necessary, we may assume that
  $n(\wt\calO)=1$ for the remaining proof of this lemma. If $A$ is split, then $\wt\calO$ is an Eichler
  order of level $\grp$, so $\Nr(\calN(\wt\calO))=F^\times$. It 
  follows from (\ref{eq:54}) and (\ref{eq:58}) that \[  \Nr(\calE)=\Nr(\wt\calE)=F^\times\neq \Nm_{K/F}(K^\times).\] 
If $A$ is ramified, then $\wt\calO$ is the unique maximal order in
$A$. Since $i(\wt B)\geq 1$, we have $\Emb(\wt{B},
\wt\calO)=\emptyset$ by  \cite[Theorem~II.3.1]{vigneras}, which in
turn implies that $\Emb(B, \calO)=\emptyset$ by the recursion.  This
contradicts the assumption that $\Emb(B, \calO)\neq \emptyset$. Therefore, $A$ cannot be ramified when
$n(\calO)\leq 2i(B)$. The
lemma is proved. 
\end{proof}


In the remaining cases, we always have $n(\wt\calO)\geq 3$ and $i(\wt
B)=1$. Thanks to (\ref{eq:58}), we may simply assume that $i(B)=1$ and
$n(\calO)\geq 3$ at
the very beginning.  In particular, $\calM(B)=O_K$.


\begin{lem}\label{lem:4.4}
  If $i(B)=1$ and $n(\calO)\geq 5$, then $\Nr(\calE)=\Nm_{K/F}(K^\times)$. 
\end{lem}
\begin{proof}
From  (\ref{eq:55}), we have 
\[\calE=\calE(\varphi, B, \calO)\subseteq \calE(\varphi, O_K,
  \calM^2(\calO)).\]
Now $n(\calM^2(\calO))\geq 3$, so it follows from
Lemma~\ref{lem:case-B=OK} that \[\Nr(\calE(\varphi, O_K,
  \calM^2(\calO)))=\Nm_{K/F}(K^\times).\]  We conclude that
$\Nr(\calE)=\Nm_{K/F}(K^\times)$ in this case. 
\end{proof}

Now we  treat the cases that $i(B)=1$ and $n(\calO)\in \{3, 4\}$. By
the assumption on $K$, there exists an
embedding  $\varphi_0: O_K\to \calO$.    From the Skolem-Noether
theorem, we may write 
\begin{equation}
  \label{eq:61}
  \varphi=z\varphi_0 z^{-1}\quad\text{for some}\quad z\in A^\times . 
\end{equation}
According to Lemma~\ref{lem:recursive}, there is a canonical decomposition
\begin{equation}
  \label{eq:62}
  \Emb(O_K, \calM^2(\calO))=\Emb(O_K, \calO)\bigsqcup \Emb(B,
  \calO). 
\end{equation}
If we define\footnote{Note that $\varphi$ is not an optimal embedding
  of $O_K$ into $\calO$, so this set cannot be denoted as
  $\calE(\varphi, O_K, \calO)$.}
\begin{equation}
  \label{eq:65}
\calC(\varphi, O_K, \calO):=\{g\in A^\times\mid \varphi(K)\cap
  g\calO g^{-1}=\varphi(O_K)\},   
\end{equation}
then $  \calE(\varphi, O_K, \calM^2(\calO))$ decomposes
into 
\begin{equation}
  \label{eq:64}
  \calE(\varphi, O_K, \calM^2(\calO))=  \calE(\varphi, B,
  \calO)\bigsqcup \calC(\varphi, O_K, \calO).  
\end{equation}
Plugging (\ref{eq:61}) into (\ref{eq:65}), we get 
\begin{equation}
  \label{eq:63}
  \begin{split}
    \calC(\varphi, O_K, \calO)=z\cdot\calE(\varphi_0, O_K, \calO). 
  \end{split}
\end{equation}
Since $n(\calM^2(\calO))\in \{1, 2\}$, we have 
\begin{equation}
  \label{eq:66}
F^\times=\Nr(\calN(\calM^2(\calO)))\subseteq  \Nr(\calE(\varphi, O_K,
\calM^2(\calO))) \subseteq F^\times. 
\end{equation}
Hence the inclusions are in fact equalities. On the other hand, from Lemma~\ref{lem:case-B=OK},
\begin{equation}
  \label{eq:68}
\Nr(\calE(\varphi_0, O_K, \calO))=\Nm_{K/F}(K^\times),   
\end{equation}
which has index $2$ in $F^\times$. 
Therefore, 
\begin{equation}
  \label{eq:67}
  \Nr(\calE(\varphi, B, \calO))\neq\Nm_{K/F}(K^\times)\quad
  \text{if}\quad \Nr(z)\in \Nm_{K/F}(K^\times).  
\end{equation}

Write
$\calO=\dangle{1,\ x_{\alpha \beta},\ \pi^r x_1,\ \pi^s x_3}$ as in
(\ref{eq:18}). For simplicity, we identify $K$ with $F(x_{\alpha\beta})$
and take $\varphi_0$ to be the identification map.  

\begin{lem}\label{lem:neq-when-n=2i+2}
  If $i(B)=1$ and  $n(\calO)=4$, then $\Nr(\calE)\neq\Nm_{K/F}(K^\times)$.   
\end{lem}
\begin{proof}
In this case, $r=2$ and $s=1$ by (\ref{eq:19}). 
Take $z=1+x_3$. We claim that
\begin{equation}
  \label{eq:71}
  \Nr(z)\in \Nm_{K/F}(K^\times)\quad \text{and}\quad zKz^{-1}\cap \calO
  =z Bz^{-1}.
\end{equation}
From (\ref{eq:16}), we have $\Tr(x_3)=0$ and
$\Nr(x_3)=-\varepsilon\pi$. Therefore,
\[\Nr(z)=(1+x_3)(1-x_3)=1-\varepsilon\pi\equiv 1\pmod{\grp}.\]
This shows that $\Nr(z)\in O_F^{\times 2}\subseteq
\Nm_{K/F}(K^\times)$ by (\ref{eq:34}). Recall that $O_K=\dangle{1,
  x_{\alpha\beta}}$. From Lemma~\ref{lem:recursive}, to show that
$zKz^{-1}\cap \calO 
  =zBz^{-1}$, it is enough to show that $zx_{\alpha\beta}z^{-1}\not\in
  \calO$. Since $\Nr(z)\in O_F^\times$, this is equivalent to show
  that $zx_{\alpha\beta}\bar{z}\not\in \calO$. Now we compute
  \begin{equation}
    \label{eq:2}
\begin{split}
zx_{\alpha\beta}\bar{z}=& ( 1+x_3) ( \alpha x_1+x_2+\beta x_3) (
                          1-x_3) \\
  =& -(1+\alpha \varepsilon )\pi + \big(\alpha + (2+\alpha \varepsilon)
     \pi\big)x_1 \\
  &+ ( 1+2\alpha \varepsilon +\varepsilon \pi ) x_2 +
  \big(\beta(1-\varepsilon\pi) -(\alpha +\pi)\big) x_3.
  \end{split}
  \end{equation}

If we write $
zx_{\alpha\beta}\bar{z}=a+bx_{\alpha\beta}+c\pi^2x_1+d\pi x_3$, then 
\[c\pi^2=(2+\alpha\varepsilon)\pi-(2\alpha+\pi)\varepsilon\alpha.\]
Since $\alpha\in \grp$, we find that $c\pi^2\equiv
2\pi\pmod{\pi^2}$, and hence $c\not\in O_F$ because $F$ is
nondyadic. This finishes the verification of our claim. 
 Now the lemma
follows from combining (\ref{eq:71}) with (\ref{eq:67}). 
\end{proof}

\begin{lem}\label{lem:4.6}
  Suppose that $i(B)=1$ and  $n(\calO)=3$. Then
  $\Nr(\calE)\neq\Nm_{K/F}(K^\times)$ if one of the following
  conditions holds:
  \begin{itemize}
  \item $A$ is split and $\abs{\grk}>5$,
  \item $A$ is ramified and $ \abs{\grk}>3$.
   \end{itemize}
\end{lem}
\begin{proof}
 From  (\ref{eq:19}), we have $r=s=1$  in this
  case. Thus $\calO=\dangle{1, x_{\alpha\beta}, \pi x_1, \pi x_3}$,
  where $x_{\alpha\beta}=\alpha x_1+x_2+\beta x_3$. Since $\alpha\in \grp$, without
loss of generality we may assume that $\alpha=0$, so
$x_{\alpha\beta}=x_2+\beta x_3$. 
  
First, suppose that $A\simeq \Mat_2(F)$ and $\abs{\grk}>5$. In this case $\varepsilon=0$,
and $1+\beta\in O_F^\times$,   so
$x_{\alpha\beta}=
\begin{bmatrix}
  0 & 1+\beta\\ \pi & 0
\end{bmatrix}$ by (\ref{eq:17}). Take $z=
\begin{bmatrix}
  1 & 0 \\ 0 & t
\end{bmatrix}$ for some $t\in O_F^{\times 2}$ with
$1-t^2\not\in\grp$. Such a $t$ exists because the number of $a\in \grk^{\times 2}$ such that $a\neq \pm 1$ is
$(\abs{\grk}-5)/2>0$. Here we have applied the assumption $-1\in \grk^{\times 2}$ in the split
case. 
We compute
\begin{equation}\label{eq:10}
    z x_{\alpha\beta} z^{-1}=\begin{bmatrix}
  0 & t^{-1}(1+\beta) \\ t\pi & 0
\end{bmatrix}=tx_{\alpha\beta}+\frac{(1-t^2)(1+\beta)}{t\pi}(\pi
x_3)\not\in \calO.
\end{equation}
The lemma in this case follows from combining (\ref{eq:10}) with
(\ref{eq:67}). 

Next suppose that $A$ is ramified and $\abs{\grk}>3$. 
  There exists $a\in \grk^\times$ such that
  $1-4a\in \grk^\times\smallsetminus \grk^{\times2}$ and
  $1-2a\neq 0$. Indeed, the number of choices for such an $a$ is at least $(\abs{\grk}-3)/2>0$.  Pick $\varepsilon\in O_F^\times$ to be any element
  such that $\varepsilon$ modulo $\grp$ is equal to $a$. Then we have
  $1-4\varepsilon\in O_F^\times\smallsetminus O_F^{\times 2}$, and
  $1-2\varepsilon\in O_F^\times$. Take this particular $\varepsilon$
  in (\ref{eq:16}) for the $F$-basis $\{1, x_1, x_2, x_3\}$ of
  $A$. Lastly, put $z=1-\varepsilon ^{-1}x_{1}$. We claim that
  \begin{equation}\label{eq:9}
    \Nr(z)=1\quad\text{and}\quad  zKz^{-1}\cap
\calO=zBz^{-1}.
  \end{equation}
The first equality  follows from a direct calculation. To prove $zKz^{-1}\cap
\calO=zBz^{-1}$, it is enough to show that $z x_{\alpha\beta}z^{-1}\not\in
\calO$. We calculate
\[\begin{split}
z x_{\alpha\beta}z^{-1}=&( 1-\varepsilon^{-1}x_1) (x_2+\beta x_3) ( 1-\varepsilon^{-1}+\varepsilon^{-1}x_1) \\ =&
(1+2\beta -\varepsilon^{-1}-\varepsilon^{-1}\beta ) x_2+( \beta -2\varepsilon^{-1}-3\varepsilon^{-1}\beta +\varepsilon^{-2}+\varepsilon^{-2}\beta ) x_3 \\=&
(1+2\beta -\varepsilon^{-1}(1+\beta)) x_{\alpha
  \beta}+
\pi^{-1}\varepsilon^{-2}(1-2\varepsilon)(1+\beta+\varepsilon\beta^2)(\pi
x_3). 
\end{split}\]
By  the above choice of $\varepsilon$ and (\ref{eq:20}), the coefficient 
$\pi^{-1}\varepsilon^{-2}(1-2\varepsilon)(1+\beta+\varepsilon\beta^2)\not\in
O_F$. Our claim is verified. Now the lemma in this  case 
follows from combining (\ref{eq:9}) with (\ref{eq:67}).
\end{proof}

 Keep the assumption  that $i(B)=1$ and $n(\calO)=3$. Let $m(B, \calO,
 \calO^\times)$ be the number of $\calO^\times$-conjugacy classes of
 optimal embeddings of $B$ into $\calO$  as in (\ref{eq:14}). 
Using the assumption that  $F$ is nondyadic,  we apply \cite[(3.12), (3.13) and
(3.15)]{Brzezinski-crelle-1990} to obtain
\begin{equation}\label{eq:76}
  m(B, \calO, \calO^\times)=
  \begin{cases}
   \left((\abs{\grk}^2-\abs{\grk})\cdot 1-2\abs{\grk}\right)/\abs{\grk}= \abs{\grk}-3 &\text{if $A$ is split},\\
   \left((\abs{\grk}^2+\abs{\grk})\cdot 1-2\abs{\grk}\right)/\abs{\grk}=\abs{\grk}-1 &\text{if $A$ is ramified}. 
  \end{cases}
\end{equation}
Assume further that  one of the following
conditions holds:
\begin{itemize}
  \item  $A$ is split and $\abs{\grk}=5$,
    \item $A$ is ramified and $ \abs{\grk}=3$.
\end{itemize}
Then up to conjugation by $\calO^\times$, there are
exactly two optimal
embeddings of $B$ into $\calO$, say $\varphi_1$ and $\varphi_2$.
Write $\varphi_i=z_i\varphi_0z_i^{-1}$ for $i=1, 2$.  We will show
that $\Nr(z_i)\not\in \Nm_{K/F}(K^\times)$ for both $i$. Since $\Nm_{K/F}(K^\times)$ is a subgroup
of index $2$ in $F^\times$, the reduced norm of  $w:=z_1z_2^{-1}$ lies
in $\Nm_{K/F}(K^\times)$, and $\varphi_2=w^{-1}\varphi_1 w$. Now 
\begin{equation}
  \label{eq:12}
\calE(\varphi_1, B, \calO)=\varphi_1(K^\times)\calO^\times\bigsqcup
\varphi_1(K^\times)w\calO^\times. 
\end{equation}
It follows from (\ref{eq:43}) that
\begin{equation}
  \label{eq:31}
  \Nr(\calE)=\Nr(\calE(\varphi_1, B, \calO))=\Nm_{K/F}(K^\times). 
\end{equation}

\begin{lem}\label{lem:4.7}
  Suppose that $i(B)=1$,  $n(\calO)=3$, $A$ is split, and
  $\abs{\grk}=5$. Then $\Nr(\calE)=\Nm_{K/F}(K^\times)$. 
\end{lem}
\begin{proof}
  Similarly as in the proof of Lemma~\ref{lem:4.6}, we take
  $x_{\alpha\beta}=\begin{bmatrix}
  0 & 1+\beta\\ \pi & 0
\end{bmatrix}$. In particular, $\alpha$ is taken to be $0$.  Put $z_1=
\begin{bmatrix}
  1 & 0 \\ 0 & 2
\end{bmatrix}
$ and $z_2=
\begin{bmatrix}
  1 & 0 \\ 0 & -2
\end{bmatrix}
$. From (\ref{eq:10}), both $z_i\varphi_0 z_i^{-1}$ are optimal
embeddings of $B$  in $\calO$. Thus to finish the proof, it is enough
to show that $z_1\varphi_0 z_1^{-1}$ and $z_2\varphi_0 z_2^{-1}$ are
not $\calO^\times$-conjugate. Suppose otherwise so that  there
exists $u\in \calO^\times$ satisfying
\begin{equation}
  \label{eq:49}
  z_2\varphi_0
z_2^{-1}=u^{-1}z_1\varphi_0 z_1^{-1}u. 
\end{equation}
Write $u=a+bx_{\alpha\beta}+c\pi x_1+d\pi x_3$ with $a, b, c, d\in O_F$. By 
(\ref{eq:42}), necessarily $a\in O_F^\times$. 
But (\ref{eq:49}) holds if and only if $z_1^{-1}uz_2=\gamma+\delta
x_{\alpha\beta}$ for some $\gamma, \delta\in F$. We compute 
\[z_1^{-1}uz_2=\begin{bmatrix}
  a+c\pi & -2b(1+\beta)-2d\pi\\ b\pi/2 & -a
\end{bmatrix}=
\begin{bmatrix}
  \gamma & \delta(1+\beta)\\ \delta\pi& \gamma
\end{bmatrix}. 
\]
Already, this implies that $c=(-2a)/\pi\not\in O_F$, contradiction to
the assumption that $u\in \calO$. Therefore, $z_1\varphi_0z_1^{-1}$
and $z_2\varphi_0z_2^{-1}$ indeed represent distinct members of
$\Emb(B, \calO)/\calO^\times$. 
 Since
$\Nm_{K/F}(K^\times)\cap O_F^\times=O_F^{\times 2}$ and $\grk=\F_5$, we find $\pm
2\not\in \Nm_{K/F}(K^\times)$. The lemma is proved. 
\end{proof}

\begin{lem}\label{lem:4.8}
  Suppose that $i(B)=1$,  $n(\calO)=3$,  $A$ is ramified, and
  $\abs{\grk}=3$. Then $\Nr(\calE)=\Nm_{K/F}(K^\times)$. 
\end{lem}  
\begin{proof}
Since $\abs{\grk}=3$, the assumption that $1-4\varepsilon\in O_F^\times\smallsetminus
O_F^{\times2}$ implies that $\varepsilon\equiv -1\pmod{\grp}$. 
  Thus if we put $z_1=x_1^{-1}$ and $z_2=x_1$, then
  $\Nr(z_i)=\varepsilon^{\mp 1}\not\in
  \Nm_{K/F}(K^\times)$. We claim that
  $z_iKz_i^{-1}\cap \calO=z_iBz_i^{-1}$ for both $i=1,2$. If not, then
  $zKz^{-1}\cap\calO=zO_Kz^{-1}$ for some $z\in \{z_1, z_2\}$. Recall that
  $\calN^0(\calO)$ acts transitively by conjugation on the set of embeddings
  $\Emb(O_K, \calO)$ (cf.~(\ref{eq:53})). Thus there exists $v\in \calN^0(\calO)$ such
  that $z \varphi_0 z^{-1}=v \varphi_0 v^{-1}$. It follows that $v^{-1}z\in K^\times$, and
  hence $\Nr(v^{-1}z)\in \Nm_{K/F}(K^\times)$. Since
  $\Nr(\calN^0(\calO))\subseteq \Nm_{K/F}(K^\times)$ by Lemma~\ref{lem:normalizer-red-norm}, we find that
  $\Nr(z)\in \Nm_{K/F}(K^\times)$ as well, contradiction to the choice
  of $z_1$ and $z_2$. 

  Next, we check that $z_1\varphi_0z_1^{-1}$ and
  $z_2\varphi_0z_2^{-1}$ are not $\calO^\times$-conjugate.  Suppose
  otherwise so that there exists $u\in \calO^\times$ with
  $z_1^{-1}uz_2\in K$. Write $u=a+bx_{\alpha\beta}+c\pi x_1+d\pi x_3$
  with $a\in O_F^\times$ and $b, c, d\in O_F$ as before. We compute
  \begin{equation}
    \label{eq:72}
    z_1^{-1}uz_2=x_1ux_1=-\varepsilon(a+c\pi)+(a+c\pi-c\varepsilon\pi)x_1+b\varepsilon
    x_{\alpha\beta}+d\varepsilon\pi x_3.
  \end{equation}
  Thus $z_1^{-1}uz_2\in K$ if and only if $a+c\pi(1-\varepsilon)=0$
  and $d=0$. Since $a\in O_F^\times$ and $1-\varepsilon\in
  O_F^\times$, we get $c\not\in O_F$ again. This contradiction shows
  that $z_1\varphi_0z_1^{-1}$ and
  $z_2\varphi_0z_2^{-1}$ indeed represent  distinct members of
$\Emb(B, \calO)/\calO^\times$. The lemma is proved. 
\end{proof}

\begin{proof}[End of the proof of Theorem~\ref{thm:main}]
  Comparing Table~\ref{tab:recursion} with
  Lemmas~\ref{lem:neq-first-case}--\ref{lem:4.8}, it is clear that we have finished the
  case-by-case study for $i(B)\geq 1$. The
   case $B=O_K$ has already been treated in
  Lemma~\ref{lem:case-B=OK}. The proof of Theorem~\ref{thm:main} is now complete. 
\end{proof}

As a by-product of our proof, we obtain the following criterion for
nonexistence of local optimal embeddings. Let $F, \calO$ and $K$ be as 
in the start of \S\ref{subsec:step-ii}, except that we only keep the
assumption that $\Emb(O_K, \calO)\neq \emptyset$ and drop the assumption
that $\Nr(\calN(\calO))=\Nm_{K/F}(K^\times)$ (see Lemma~\ref{lem:normalizer-red-norm}). 
Write $E_\ur$ for the unique unramified quadratic field
  extension of $F$.

\begin{cor}\label{cor:opt-embed-nonempty}
    Let $L/F$ be a semisimple
  quadratic extension, and $R$ be an order in $L$. Assume that
  $\Hom_F(L, A)\neq \emptyset$, that is, $L\neq
  F\times F$  if $A$ is ramified.
  Then $\Emb(R,
\calO)=\emptyset$ if and only if one of the following holds: 
\begin{enumerate}
\item $n(\calO)< 2i(R)$ and $A$ is ramified,
\item $n(\calO)=2i(R)$, $A$ is ramified, and $L/F$ is ramified, 
\item $n(\calO)=2i(R)+1$, $L=E_\ur$,  and $A=\Mat_2(F)$,
\item $n(\calO)=2i(R)+1$, $L=K$, $A=\Mat_2(F)$ and $\abs{\grk}=3$,
\item $n(\calO)=2i(R)+2$, and either 
  $L=F\times F$ or $L=E_\ur$, 
\item $n(\calO)\geq  2i(R)+3$ and $L\neq K$. 
\end{enumerate}
\end{cor}
See \cite[Theorem~3.10]{Brzezinski-crelle-1990} for the case
$n(\calO)=2$, which holds even if $F$ is dyadic. 
\begin{proof}
Applying Lemma~\ref{lem:recursive} recursively to $(R, \calO)$, we
eventually obtain a new pair $(\wt R, \wt \calO)$ for which $\Emb(R,
\calO)=\Emb(\wt R, \wt\calO)$. From Table~\ref{tab:recursion}, the
discussion is again separated into four cases. 

First, suppose that $n(\calO)\leq 2i(R)$ so that $n(\wt\calO)\in \{1, 2\}$
and $i(\wt R)\geq 1$. We further divide it into two subcases according
to whether $n(\wt\calO)$ is equal to $1$ or $2$. Suppose that $n(\wt\calO)=1$.  If $A$ is split, then $\wt\calO$ is an Eichler
order of level $\grp$, so $\Emb(\wt R, \wt \calO)\neq \emptyset$ by
\cite[Theorem~II.3.2]{vigneras}. If $A$ is ramified, then $\wt\calO$ is
the unique maximal order, so $\Emb(\wt R, \wt
\calO)=\emptyset$ by \cite[Theorem~II.3.1]{vigneras}. Next, suppose
that $n(\wt\calO)=2$. From 
\cite[(3.17)]{Brzezinski-crelle-1990}, $\Emb(\wt R, \wt
\calO)\neq \emptyset$ if and only if one of the following conditions
holds:
\begin{itemize}
\item  $A$ is split,
\item $A$ is ramified, $L=E_\ur$, and $i(\wt R)=1$. 
\end{itemize}
Note that $(i(\wt R), n(\wt\calO))=(1,2)$ if and only if
$n(\calO)=2i(R)$. This shows that when $n(\calO)\leq 2i(R)$, we have
$\Emb(R, \calO)=\emptyset$ if and only if either condition (1) or (2)
holds. 

Now suppose that $n(\calO)=2i(R)+1$ so that $(i(\wt R),
n(\wt\calO))=(1, 3)$. We further divide it into two subcases according
to whether $L=K$ or not. First, suppose that $L\neq K$. Then $\Emb(\wt
R,\wt\calO)=\Emb(O_L, \calM^2(\wt\calO))$ by
Lemma~\ref{lem:recursive}. Since $n(\calM^2(\wt\calO))=1$, we find
that $\Emb(O_L, \calM^2(\wt\calO))=\emptyset$ if and only if
$L=E_\ur$ and $A=\Mat_2(F)$. This gives part (3) of the
corollary. Next, suppose that $L=K$. From (\ref{eq:76}), we
immediately see that $\Emb(\wt
R,\wt\calO)=\emptyset$ if and only if $A=\Mat_2(F)$ and
$\abs{\grk}=3$. This gives part (4) of the
corollary.

Next, suppose that $n(\calO)=2i(R)+2$ so that $(i(\wt R),
n(\wt\calO))=(1, 4)$. If $L=K$, we have seen in the proof of
Lemma~\ref{lem:neq-when-n=2i+2} that $\Emb(\wt R, \wt \calO)\neq
\emptyset$. Suppose  that $L\neq K$ so that $\Emb(\wt
R,\wt\calO)=\Emb(O_L, \calM^2(\wt\calO))$ again.  Since
$n(\calM^2(\wt\calO))=2$, it follows from
\cite[(3.14)]{Brzezinski-crelle-1990} that $\Emb(O_L,
 \calM^2(\wt\calO))\neq \emptyset$ if and only if $L/F$ is ramified. This gives part (5) of the
 corollary.

Lastly, suppose that  $n(\calO)\geq  2i(R)+3$ so that $i(\wt R)=1$ and
$n(\wt\calO)\geq 5$. From Lemma~\ref{lem:recursive}, $\Emb(\wt R,
\wt\calO)\subseteq \Emb(O_L, \calM^2(\wt\calO))$. Thus if $L\neq K$,
then $\Emb(\wt R,
\wt\calO)=\emptyset$ since $O_L$ does not embed into
$\calM^2(\wt\calO)$ by Lemma~\ref{lem:unique-ram-ext}. If $L=K$, then
according to  \cite[(3.13) and (3.15)]{Brzezinski-crelle-1990}, we have 
\[m(\wt R,
\wt\calO, \wt\calO^\times)=\frac{1}{\abs{\grk}}\left(\abs{\grk}^2\cdot
  2\abs{\grk}-2\abs{\grk}\right)=2(\abs{\grk}^2-1)>0. 
\]
This gives part (6) of the corollary and completes the proof. 
\end{proof}

\section{Examples}
\label{sec:an-exam}
\numberwithin{thmcounter}{section}
In this section, we construct a  family of concrete examples where
$B\subset K\subseteq \Sigma_\scrG$ and $\Emb(B, \calO)\neq \emptyset$
for every $\calO\in \scrG$ (i.e.~$B$ is not optimally selective). 

Let $p\in \bbN$ be a prime with  $p\equiv 1\pmod{4}$. Fix an integer
$n\geq 3$,  and put $r=\fl{n/2}$ and $s=\fl{(n-1)/2}$ as in
(\ref{eq:33}). Pick  $t \in \Z\cap\Z_p^{\times2}$, i.e.~$t\in \Z$ and
is a quadratic residue\footnote{As a convention, we exclude the case $t\equiv
  0\pmod{p}$ when discussing quadratic (or quartic) residues or non-residues modulo $p$.} modulo $p$. We define two
orders in $A=\Mat_2(\Q)$:
 \begin{align}
   \label{xeq:1}
  \calO:=& \Z
  \begin{bmatrix}
    1 & 0 \\ 0 & 1
  \end{bmatrix}+\Z  \begin{bmatrix}
    0 & 1 \\ p & 0
  \end{bmatrix}+\Z \begin{bmatrix}
    p^r & 0 \\ 0 & 0
  \end{bmatrix}+\Z  \begin{bmatrix}
    0 & p^s \\ 0 & 0
  \end{bmatrix},\\
     \calO':=& \Z
  \begin{bmatrix}
    1 & 0 \\ 0 & 1
  \end{bmatrix}+\Z  \begin{bmatrix}
    0 & t \\ p & 0
  \end{bmatrix}+\Z \begin{bmatrix}
    p^r & 0 \\ 0 & 0
  \end{bmatrix}+\Z  \begin{bmatrix}
    0 & p^s \\ 0 & 0
  \end{bmatrix}.
 \end{align}
From (\ref{eq:18}),  both
$\calO_p$ are $\calO_p'$ are Bass $\Z_p$-orders in $\Mat_2(\Q_p)$ with 
\begin{equation}
  \label{eq:73}
e_p(\calO)=e_p(\calO')=0,\quad\text{and}\quad
n_p(\calO)=n_p(\calO')=n. 
\end{equation}
Indeed, $\calO_p$ is precisely the order in (\ref{eq:18}) with
$\pi=p$,  $\alpha=\beta=0$. Similarly, we have taken
$\alpha=0$ and $\beta=t-1$ for $\calO_p'$.
A direct calculation shows that both $\calO$ and $\calO'$ have index
$p^n$ in $\Mat_2(\Z)$. Hence
\begin{equation}
  \label{eq:74}
\calO_\ell=\calO_\ell'= \Mat_2(\Z_\ell)\qquad \text{for every prime }
  \ell\neq p.   
\end{equation}
 Pick $u_p\in \Z_p^\times$ such that
$u_p^2=t$ and put $h_p:=
\left[\begin{smallmatrix}
u_p & 0 \\ 0 & 1
\end{smallmatrix}\right]$. 
Then $h_p\calO_ph_p^{-1}=\calO_p'$. Therefore, $\calO$ and $\calO'$ belong
to the same genus.  Let $\scrG$ be the genus  of $\calO$ and $\calO'$,
and $\Sigma_\scrG$ be the spinor genus field of $\scrG$.   For
simplicity, write $K=\Q(\sqrt{p})$.

\begin{lem}
  $\Sigma_\scrG=K=\Q(\sqrt{p})$. In particular, $\abs{\Tp(\scrG)}=2$. 
\end{lem}
\begin{proof}
From Lemma~\ref{lem:unique-ram-ext}, $K_p$ is the unique quadratic
extension of $\Q_p$ such that $O_{K_p}$ embeds into $\calO_p$. Since $p\equiv 1\pmod{4}$, we have $-1\in
\Z_p^{\times2}$. Thus $K\subseteq \Sigma_\scrG$ by
Proposition~\ref{prop:description-Sigma}. On the other hand,
$\Sigma_\scrG/\Q$ is the compositum of its quadratic subextensions,
but $K/\Q$ is
the unique quadratic extension unramified outside  $p$. We conclude
that $\Sigma_\scrG=K$. 

From (\ref{eq:8}), we have
$\abs{\SG(\scrG)}=[\Sigma_\scrG:\Q]=2$. Since $A=\Mat_2(\Q)$, which
clearly satisfies the Eichler condition, $\SG(\scrG)$ is canonically
identified with $\Tp(\scrG)$  by
Remark~\ref{rem:spinor-genus=type}. Therefore, $\abs{\Tp(\scrG)}=2$. 
\end{proof}
\begin{lem}
The orders $\calO$ and $\calO'$ are of the same type if
  and only if $t\in \Z_p^{\times 4}$, that is, $t$ is a quartic
  residue modulo $p$.  
\end{lem}
\begin{proof}
First, the hereditary closures of both $\calO_p$ and $\calO_p'$ coincide with $\begin{bmatrix}
  \Z_p & \Z_p\\ p\Z_p & \Z_p
\end{bmatrix}$.  Put $\scrO=
\begin{bmatrix}
  \Z & \Z \\ p\Z & \Z
\end{bmatrix}$.
If $g\calO g^{-1}=\calO'$ for some $g\in \GL_2(\Q)$, then necessarily $g\scrO
g^{-1}=\scrO$, that is, $g\in \calN(\scrO)$. It is well known that
\[\calN(\scrO)=\Q^\times\scrO^\times\bigsqcup \Q^\times\scrO^\times
  \begin{bmatrix}
    0 & 1 \\ p & 0
  \end{bmatrix}.
\]
On the other hand,  $\left[\begin{smallmatrix}
 0 & 1 \\ p & 0
\end{smallmatrix}\right]\in \calN(\calO)$. Indeed, clearly $\left[\begin{smallmatrix}
 0 & 1 \\ p & 0
\end{smallmatrix}\right]\in \calN(\calO_\ell)$ for each prime
$\ell\neq p$. Moreover, from (\ref{eq:36}) we have $\left[\begin{smallmatrix}
 0 & 1 \\ p & 0
\end{smallmatrix}\right]\in \calN(\calO_p)$ since $x_{\alpha\beta}=x_2=\left[\begin{smallmatrix}
 0 & 1 \\ p & 0
\end{smallmatrix}\right]$ for $\calO_p$. Therefore, if there exists
$g\in \GL_2(\Q)$ such that $\calO'=g\calO g^{-1}$, then it can
be taken inside $\scrO^\times$.

Suppose that there exists $g\in\scrO^\times$ such that $g\calO g^{-1}=\calO'$. Then we have
$h_p^{-1}g\in \calN(\calO_p)$. From
Lemma~\ref{lem:normalizer-red-norm}, \[\Nr(h_p^{-1}g)\in
  \Nr(\calN(\calO_p))\cap \Z_p^\times=\Z_p^{\times 2}.\]
Since $\Nr(g)\in \Nr(\scrO^\times)=\{\pm 1\}$ and $p\equiv 1\pmod{4}$, we get $u_p=\Nr(h_p)\in \Z_p^{\times
  2}$, which implies that $t=u_p^2\in \Z_p^{\times 4}$. 

Next, suppose that $t\in \Z_p^{\times 4}$. Then the equation $x^2=u_p$
has a solution $v_p\in\Z_p^\times$. From \cite[\S6.1]{lang-ell-func}, the
canonical map $\SL_2(\Z)\to \SL_2(\zmod{p^{r+1}})$ is surjective. In particular,
there exists $g\in \SL_2(\Z)$ such that
\begin{equation}\label{eq:75}
g\equiv    \begin{bmatrix}
    v_p & 0 \\ 0 & v_p^{-1}
  \end{bmatrix} \pmod{p^{r+1}}.
\end{equation}
We claim that $g\calO g^{-1}=\calO'$. It is enough to show that
$g\calO_\ell g^{-1}=\calO'_\ell$ for every prime $\ell$ (including
$\ell=p$). If $\ell\neq p$, this is clear from (\ref{eq:73}). At the
prime $p$, observe that $\calO_p\supseteq \grO_p:=\Z_p+p^{r+1}\Mat_2(\Z_p)$. 
The choice of $g$ in (\ref{eq:75}) guarantees that $h_p^{-1}g\in
\grO_p^\times\subseteq \calO_p^\times$, which implies that $g\calO_p g^{-1}=\calO_p'$. This
finishes the verification of our claim and the proof of the lemma. 
\end{proof}
\begin{ex}
  Suppose that $p\equiv 5\pmod{8}$. Then $-1\in
  \Z_p^{\times2}\smallsetminus \Z_p^{\times 4}$. Thus if we put
  $t=-1$, then $\Tp(\scrG)$ is represented by $\calO$ and $\calO'$. 
\end{ex}

\begin{prop}\label{prop:example}
  Suppose that $t\not\in \Z_p^{\times 4}$ so that $\{\calO, \calO'\}$
  is a complete set of representatives for $\Tp(\scrG)$. 
  Let $B$ be an order in
$\Q(\sqrt{p})$. Suppose that  $n<2i_p(B)+3$,  and $p\neq 5$ if
$n=2i_p(B)+1$. 
Then both  $\Emb(B, \calO)$ and $\Emb(B, \calO')$ are nonempty. In other words, $B$ is not optimally selective for the genus $\scrG$. 
\end{prop}
\begin{proof}
From (\ref{eq:74}), $\Emb(B_\ell, \calO_\ell)\neq \emptyset $ for
every prime  $\ell\neq p$. According to 
Corollary~\ref{cor:opt-embed-nonempty}, $\Emb(B_p,
\calO_p)$ is nonempty as well. The proposition follows directly from
Theorem~\ref{thm:main}. 
\end{proof}


Lastly, we consider the global number of optimal embeddings up to
conjugation as in (\ref{eq:14}). Let $\grF$ be a number field, and $\grA$ be a quaternion $\grF$-algebra
satisfying the Eichler condition. 
Let $\grO, \grO'\subset \grA$ be two orders in the same genus $\grG$. Suppose
that $e_\grp(\grO)\neq 0$ for every finite prime $\grp$ of $\grF$. Let 
  $\grB$ be an $O_{\grF}$-order in a quadratic field extension $\grK/\grF$ with $\Emb(\grB_\grp, \grO_\grp)\neq \emptyset$ for every
  $\grp$. Suppose that either $\grB$ is not optimally selective for
  $\grG$  or both $\grO$ and $\grO'$ are optimally selected by $\grB$. Then
  \begin{equation}\label{eq:13}
 m(\grB, \grO, \grO^\times)=m(\grB, \grO', \grO'^\times).    
  \end{equation}
 See
\cite[Theorem~31.1.7]{voight-quat-book} for the proof in the case of
Eichler orders and \cite[Proposition~2.15]{xue-yu:spinor-class-no} for
the proof in general.  Naturally, one asks whether the equality
(\ref{eq:13}) still holds true if 
$e_\grp(\grO)$ is allowed to be zero at some finite prime $\grp$. From
\cite[Proposition~2.15]{xue-yu:spinor-class-no},  inequality is
possible\footnote{But so far we do not know any such examples.} only if
$\grK\subseteq \Sigma_\grG$ and $\grB$ is not optimally selective
for $\grG$. Our
family of examples fit this description perfectly, so we ask the
following concrete question.   
\begin{que}
Under the assumption of Proposition~\ref{prop:example}, do we  have
  \begin{equation}
m(B, \calO, \calO^\times)=m(B, \calO', \calO'^\times)?
  \end{equation}
\end{que}

\def\cprime{$'$}

\end{document}